\documentclass{amsart}
\usepackage[utf8]{inputenc}
\usepackage[T1]{fontenc}
\usepackage[english,francais]{babel}%\selectlanguage{french}
\usepackage{color}
\usepackage{graphicx}
\DeclareGraphicsExtensions{.pdf,.png,.jpg}

\newcommand{\myblu}[1]{#1}%{\textcolor{blue}{#1} }
\newcommand{\myred}[1]{#1}%{\textcolor{red}{#1} }
\newcommand{\mycya}[1]{#1}%{\textcolor{cyan}{#1} }
\newcommand{\mypin}[1]{#1}%{\textcolor{magenta}{#1} }
\newcommand{\mygreen}[1]{#1}%{\textcolor{green}{#1} }
\newcommand{\mycomment}[1]{} %  { ***~{\bf #1}~***}}
\usepackage{amssymb}
\usepackage{amsfonts}
\usepackage{amsmath}
\usepackage{amsthm}
\newtheorem{proposition}{Proposition} % often included in packages, eg svmult.cls
\newcommand\beq{\begin{equation}}
\newcommand\eeq{\end{equation}}
\newcommand{\e}{\varepsilon}

\renewcommand{\to}{\rightarrow}

 \newcommand{\bolda}{\boldsymbol{a}}
\def\Dcal{\mathcal{D}}

\def\Fcal{\mathcal{F}}

\def\Ical{\mathcal{I}}

\newcommand{\FF}{\mathbb{F}}

\newcommand{\NN}{\mathbb{N}}

\newcommand{\QQ}{\mathbb{Q}}
\newcommand{\RR}{\mathbb{R}} \newcommand{\R}{\mathbb{R}} %\def\R{\mathbb{R}}

\newcommand{\ZZ}{\mathbb{Z}}

\begin{document}

%------------------------------------------------------------
\title[Construction and % numerical
performance of kinetic schemes]{Construction and % numerical
performance \\ of % vectorial
kinetic schemes \\ for \emph{linear}
% symmetric-hyperbolic \\
systems of conservation laws
}\thanks{Thanks to ANR for its grant 15-CE01-0013 to the SEDIFLO project.
Thanks to Lucas Br\'elivet for its help as an ENPC student in numerical simulations.}

%%%%%%%%%%%%%%%%%%%%%%%%% \shorttitle{}

\author{Emmanuel Audusse \and Sébastien Boyaval \and Virgile Dubos \and Minh Le}
\address{Université Paris 13, Villetaneuse, France}
\address{LHSV, ENPC, Institut Polytechnique de Paris, EDF R\&D, Chatou (sebastien.boyaval@enpc.fr) \& Inria, France}
%\address{Inria, France}
\address{LHSV, ENPC, Institut Polytechnique de Paris, EDF R\&D, Chatou, France}

%------------------------------------------------------------

\begin{abstract}
We describe a methodology to build \emph{vectorial} kinetic schemes,
targetting the numerical solution of \emph{linear} symmetric-hyperbolic systems of conservation laws
-- a minimal application case for those schemes. % for many physically-relevant applications.
Precisely, we fully detail the construction of %vectorial
kinetic schemes that satisfy a discrete equivalent to a convex extension
(an additional non-trivial conservation law) % not balance, no dissipation !
% for the sake of stability ?
of the target system -- the (linear) % particular $3\times 3$
acoustic % system and (or equivalently, the one-dimensional solution to)
and elastodynamics systems, specifically --.
Then, % in absence of a priori guarantee of convergence
we evaluate \emph{numerically} the convergence of various possible kinetic schemes toward smooth solutions,
in comparison with standard finite-difference and finite-volume discretizations on % uniformly-refined
Cartesian meshes. % natural FD schemes -- and Muscl-FV schemes ?? --
\mygreen{Our numerical results confirm the interest of ensuring a discrete equivalent to a convex extension,
and show the influence of remaining parameter variations in terms of error magnitude, % not convergence rate,
both for ``first-order'' and ``second-order'' kinetic schemes: the parameter choice with largest CFL number
(equiv., smallest spurious diffusion in the equivalent equation analysis) has the smallest discretization error.}
\end{abstract}

% \begin{resume}
% But Stratégie Ici
% \end{resume}

% \begin{keywords}
%  symmetric-hyperbolic PDEs ; % Friedrichs symmetric-hyperbolic systems = linear
%  numerical solutions ;
%  Finite-Volume schemes ;
%  kinetic schemes
% \end{keywords}

%\begin{AMS} 65B35 \end{AMS}

\maketitle

\section{Introduction}
\label{sec:intro}

%%% PDE's theory

Systems of Conservation Laws (SCL) with smooth fluxes $F_a(q) \in \RR^L$, $1 \leq a \leq d$, 
\begin{equation}\label{eq:scl}
\partial_t q^l + \sum\limits_{a=1}^d \partial_a F_a^l(q) = 0  % B^l
\quad 1\leq l\leq L
\end{equation}
that satisfy an additional non-trivial conservation law 
\begin{equation}
\label{eq:additional}
    \partial_t \eta(q) + \sum\limits_{a=1}^d \partial_a G_a(q) = 0
\end{equation}
where $\eta(q)$ is a scalar $C^2(\RR^L)$ functional strictly % at least in some open subset of phase space
convex in % the conservative variable
$q$ % sometimes termed ``mathematical entropy''
and $G_a$ are smooth ``entropy'' fluxes, are widely spread % modelling tools/PDEs especially
in the physical sciences. They have a symmetric-hyperbolic formulation 
so solutions $q \in C^0_t\left([0,T),[H^s(\RR^d)]^L\right)$ are unequivocally % well
defined on % sufficiently 
small times $T$ given initial conditions $q(t=0)=q^0 \in [H^s(\RR^d)]^L$, $s>1+\frac{d}2$ \cite{dafermos-2000,benzonigavage-serre-2007}.
The particular case of \emph{linear} symmetric-hyperbolic systems $F_a^l(q)=\myred{\sum_{m=1}^L}A_a^{lm}q_m$, $A_a^{lm} = A_a^{ml}$, % ($1\le m,l\le L$), 
where $\eta(q)=\tfrac12|q|^2$ and $G_a(q)=\tfrac12\myred{\sum_{m,l=1}^L}A_a^{lm}q_mq_l$, % \frac12
is a minimal benchmark to evaluate % the stability and
the accuracy of numerical approximations to $q(t\equiv x_0,x_1,\dots,x_d)$.
Noticeably, % the Cauchy problem for
solutions are then well-defined whatever $T>0$ and % i.e. it is global-in-time
whatever $s\in\RR$.
The same remarkable properties % straightforwardly
extend to \emph{Friedrichs symmetrizable} linear systems i.e. linear systems \eqref{eq:scl}
for which a symmetric, positive-definite matrix $S_0$ exists so that $S_0A_a = A_a^TS_0$
($S_0^{\frac12}q$ is solution to a symmetric-hyperbolic system then, with $\eta(q)=q^TS_0q/2$).
% \cite[Déf~1.2, p.13]{benzonigavage-serre-2007}.
% For nonlinear systems, singularities necessarily appear after some finite time (in a given coordinate system).
% Beyond 1D, various possible types of singularities exist (cf Alinhac, Sideris, Christodoulou, Speck after John),
% some solutions may remain piecewise smooth across a smooth boundary (Majda, Trakhinin)
% otherwise one still lacks a general uniqueness framework.
% Note that a number of ``singularities'' can be removed simply by changing the coordinate system (Temple).
% In general, it seems unclear when/how physics is still missing for uniqueness.
%
Friedrichs symmetrizable linear systems are physically relevant: here, we concentrate % restrict
on the linear flux case having in mind % applications to
the % numerical
evaluation of discretization methods for linear \emph{elastodynamics}.

%%% end of PDE's theory

\medskip

%%% numerical methods

Many numerical methods have been proposed to compute % numerically, some
solutions $q(t,x_1,\ldots,x_d)$ to % Cauchy problems when
\eqref{eq:scl}. % are complemented by IC
% in the linear flux case naturally (sometimes specifically) and also
% including beyond maximal ``smooth'' existence time (for nonlinear cases)
% We consider methods where time-evolution is discretized using explicit RK (Euler forward) schemes. <===
Various Finite-Difference (FD) \cite{MR2380849} % on Cartesian grids
and Finite-Volume (FV) methods \cite{leveque-2002,Godlewski-Raviart2021}
have proved able to approximate well solutions. % smooth, and even non-smooth sometimes
Yet, as demonstrated e.g. in % the recent monograph 
\cite{hesthaven-2018}, % mostly about the 1D case ($d=1$)
standard techniques leave room to reach accuracy at a reasonable computational price, especially for multidimensional solutions to nonlinear systems. % even well-def smooth case ?
See also \cite{s10915-022-02076-3} for a recent study of % the convergence rate of 
Discontinuous Galerkin (DG) methods 
toward solutions to linear (symmetric-hyperbolic) systems,
or \cite{FVLongTimeBehaviorHyperbolic} for the long-time behaviour of FV schemes (the particular DG-0 case).

%%% end of numerical methods

\smallskip

In this work, we study the construction and numerical performance of % a Boltzmann method that produces so-called
\emph{vector(ial) kinetic schemes} as defined in Section~\ref{sec:scheme},
for % the numerical approximation of
linear symmetric-hyperbolic systems of conservation laws, % first
with a view to founding the development of the kinetic approach to further cases.
%  like elastodynamics
After recalling some elements of analysis % standard (and less standard) 
about vectorial kinetic schemes % for linear symmetric-hyperbolic systems of conservation laws
in Section~\ref{sec:scheme}, 
we develop in Section~\ref{sec:application} \emph{a construction method} 
for the paradigmatic % case of 
(linear) acoustics and elastodynamics systems.
Finally, the construction method is first based on ensuring a discrete equivalent to a convex extension of the SCL at, as well as preserving some (rotational) symmetries of the SCL. Next, it leaves one with various choices for parameters. In Section~\ref{sec:numerical}, we numerically evaluate the performance of specific schemes (with specific parameter values) in comparison with standard FD and FV schemes.

\mygreen{Our numerical results first confirm the interest of ensuring a discrete equivalent to a convex extension,
both for ``first-order'' and ``second-order'' kinetic schemes: ensuring a discrete equivalent to a convex extension through adequate parameter choices appears as a sufficient and necessary condition of stability in our numerical tests. Moreover, both first and second order schemes do have first and second order convergence rates. Last, the remaining parameter choices (within the numerical stability region) show a strong influence on the magnitude of the discretization error. The parameter choice with largest CFL number
(equiv., the one with smallest spurious diffusion in the equivalent equation analysis) has the smallest discretization error, smaller than standard second order FD and FV schemes.}
% notation conventions *** In the sequel, we use Einstein's convention for double indices. *** % ($m$ or $\alpha$).
% We denote $\nabla$ the gradient operator on Euclidean space, and the divergence.

Our numerical results are an incentive to further studies about the convergence % conditions
and the error of kinetic schemes for (symmetric-hyperbolic) SCL

\section{Vectorial kinetic schemes}
\label{sec:scheme}

Kinetic schemes use $Q$ kinetic variables $f_{\zeta}$, $\zeta \in\{0,1\dots,Q-1\}$,
which are solutions to (discretized) Boltzmann-BGK equations with speeds %celerities
$c_\zeta \in\RR^d$ and a relaxation time $\e\ll T$. Here we consider \emph{vectorial} (or vector) kinetic schemes, with vectorial unknown $f_{\zeta}\in\RR^L$ solutions to % $L$
Boltzmann-BGK equations %\cite{bouchut-2003} $\zeta \in\{0,1\dots,Q-1\}$
\begin{equation} \label{eq:Vbgk}
\partial_t f_{\zeta,l} + \sum\limits_{a=1}^d c_\zeta^a  \partial_a f_{\zeta,l}  = \dfrac{\bar f_{\zeta,l}(Mf)- f_{\zeta,l}}{\e}
\end{equation}
for $l\in\{1,\dots,L\}$, using ``Maxwellians'' $\bar f_{\zeta,l}$ functions of moments of $f_{\zeta,l}$ like 
% $M_0f \in \RR^L = \int d\zeta $
\begin{equation}
\label{eq:M0}
M_0f_l := \sum\limits_\zeta % {\zeta=0}^{Q-1}
f_{\zeta,l} \,.
\end{equation}
Here, we % set $c_0=0_{\RR^d}$ by convention, and positive Cartesian norm $|c_\zeta|>0$ for $\zeta>0$,
consider $Q\ge d$ such that the speeds %celerities
$c_\zeta \in\RR^d$ define \emph{lattices} of $\RR^d$ % https://en.wikipedia.org/wiki/Lattice_(group)
% inducing grids/meshes that are _regular_ tilings with a ``primitive cell''
% i.e. tesselations of $\R^d$ into polygons/polyhedra with a symmetry group that acts on vertices, edges and tiles ??
parametrized by a uniform time-step $k=T/N>0$ ($N\in\NN^*$): for all % speed $c_\zeta$,
$\zeta=0\ldots Q-1$ it holds $\boldsymbol{a}_{ij}+kc_\zeta = \boldsymbol{a}_{i'j'}$ for some $i',j'\in\ZZ$.
%
% Then, given one fixed $k$, kinetic variables $f^n$, $n=0\ldots N-1$, can be computed
% independently one another on each lattice of  points $a_{ij}\in\RR^2$ ($i,j\in\ZZ$).
%
Such (vectorial) kinetic schemes are sometimes referred to as (vector) ``Lattice Boltzmann Methods'' (LBMs).

In this work, we consider square lattices % p4m, (*442), [4,4]
equiv. \emph{Cartesian grids} % unequivocally
with $c_0=0$
\begin{equation}
\label{CFLBM}
\Delta=\inf_{{ij}\neq{i'j'}}|\boldsymbol{a}_{ij}-\boldsymbol{a}_{i'j'}|=k\min_{\zeta>0}|c_\zeta|\,, % >0    
\end{equation}
$\boldsymbol{a}_{ij}$ has coordinates $(i\Delta,j\Delta)$.
Denoting $f_{\zeta,l}^{n}(\boldsymbol{a}_{ij})$ the kinetic variables for computing % numerical approximations
$q_{ij}^n\approx q(kn,\boldsymbol{a}_{ij})$, $n=1\ldots N$,
% (Adaptation to more general meshes is possible but requires additional numerical steps, with additional discretization error then, see e.g. \cite{baty-drui-helluy}.)
we are interested in kinetic schemes as e.g. \cite{Coulette2020,baty-drui-helluy},
which read on 2D grids of points $\boldsymbol{a}_{ij}$ with coordinates $(\Delta i,\Delta j)$ % and
\begin{equation}
\label{eq:Vbgknp1}
f_{\zeta,l}^{n+1}(\boldsymbol{a}_{ij}+kc_\zeta) = (1-\omega) f_{\zeta,l}^n(\boldsymbol{a}_{ij}) + \omega \bar f_{\zeta,l}(M_0f^n(\boldsymbol{a}_{ij}))
\end{equation}
where for all $\zeta$ and $ij$,
$\boldsymbol{a}_{i'j'}=\boldsymbol{a}_{ij}+kc_\zeta$ is uniquely defined on the lattice,
while $\omega \in (1,2)$ ranges for ``first-order'' schemes when $\omega=1$ to ``second-order'' schemes when $\omega=2$.
% \omega=\frac{k}{\e+k\theta} $\theta=\frac12,1$) -> 1 or 2 si \e << k... but !
The latter schemes coincide % for instance
with a % the following
two-step % fractional-step time-discretization method
time-integration of \eqref{eq:Vbgk}:
\begin{eqnarray}
\label{projectionStep}
& f_{\zeta,l}^{n+\frac12}(\boldsymbol{a}_{ij}) = (1-\omega) f_{\zeta,l}^n(\boldsymbol{a}_{ij}) + \omega \bar f_{\zeta,l}(M_0f^n(\boldsymbol{a}_{ij}))
\\
\label{transportStep}
& f_{\zeta,l}^{n+1}(\boldsymbol{a}_{ij}+kc_\zeta) = f_{\zeta,l}^{n+\frac12}(\boldsymbol{a}_{ij})
\end{eqnarray}
resulting from a Projection--Transport operator--splitting on $t\in[nk,(n+1)k)$.

The approximation of \eqref{eq:scl} by \eqref{eq:Vbgknp1} is not yet fully understood, even in the case of (smooth) solutions to \emph{linear} SCL, to our knowledge. \mypin{Various \emph{stability conditions} have been proposed in the literature
\cite{bouchut-2003,HELIE_Romane_2023_ED269}, as well as elements of \emph{error analysis} (i.e. heuristics % non-(fully)-rigorous arguments
for a future rigorous proof) \cite{dubois2008equivalent,PhysRevE.93.033310,Bellotti_Graille_Massot_2021,m2an220088},
% like \emph{equivalent equations} satisfied by moments $M_0f^n$, after various possible asymptotic expansions
which presently guide the construction of a kinetic scheme. But we are not aware of sufficient conditions, about % -- here
finitely-many % -- 
kinetic speeds or ``celerities'' $c_\zeta$ and associated % -- here 
vector % -- 
Maxwellians $\bar f_\zeta$, that guarantee a priori the \emph{convergence} of % (interp/reconstruction based on) solutions to 
\eqref{eq:Vbgknp1} to \eqref{eq:scl}. % in some norm -- adequate for some solutions --, at some rate
That is why here, focusing on the most simple SCL, the linear acoustics/elastodynamics systems in particular,
we explore some sensible choices of speeds and Maxwellians \emph{numerically} (see Section~\ref{sec:numerical}).
To that aim, let us first recall what are sensible choices presently.}

\smallskip

First, the (moments $M_0f^n$ of the) solutions to \eqref{eq:Vbgknp1} are known to solve equivalent equations consistent with \eqref{eq:scl} % at various rates
provided the Maxwellians $\bar f$ % \in C^2(\RR^L,\RR^{Q\times L})
satisfy % \mathcal{C}^2
\begin{align}
\label{req1}
M_0 \bar f(q) &=q
\\
\label{req2}
M_a \bar f(q) &= F_{a}(q)
\end{align}
for all $q\in \RR^L$, on denoting $M_0$ the operator defined in \eqref{eq:M0}, and $M_a:\RR^{Q\times L}\to\RR^L$
\begin{equation}
\label{eq:Ma}
M_a f_l := \sum\limits_{\zeta=0}^{Q-1}c_{\zeta}^a f_{\zeta,l}
\end{equation}
see e.g. \cite{m2an220088} for a recent statement of that result.
Thus, \eqref{req1} and \eqref{req2} are standardly required in the construction of kinetic schemes for \eqref{eq:scl}.
Choosing speeds and Maxwellians that satisfy \eqref{req1} and \eqref{req2} is sensible,
it provides one with $(d+1)L$ % independent ??
equations % as soon as c_\zeta\neq1 for all \zeta
for $QL$ unknowns $\bar f$.
% Requirement i) consists in $L$ equations, and requirement ii) consists in $d\times L$ equations.
But \eqref{req1} and \eqref{req2} are not enough in general to define uniquely the Maxwellians $\bar f$
(e.g. when using standard lattices with $Q>d+1$ celerities % *** EXPLAIN MORE !? TBC ***
and linear Maxwellians). 
Moreover, letting alone the fact that \eqref{req1} and \eqref{req2} are not enough to guarantee convergence a priori,
\eqref{req1} and \eqref{req2} do also not obviously provide one with a (simple, pointwise) error analysis % details !!
in the case of smooth solutions -- unlike standard FD and FV schemes as Yee's scheme or the upwind scheme
(see Sections~\ref{sec:numerical}, \ref{sec:FD} and \ref{sec:FV}).
Now, an error analysis % which can be called consistency in the FD sense
is one important step in proving convergence, independent of stability, % of equivalent FD scheme: eg in von Neumann sense
which still needs studying \cite{Bellotti_Graille_Massot_2021}. % in conclusion: announced, not published
So, with a view to complementing \eqref{req1} and \eqref{req2} here, let us now provide some elements for
a detailed error analysis in the case of smooth solutions, inspired by the heuristics \cite[Proposition 2.1]{bouchut-2003} leading to an equivalent equations when $\omega=1$.
\begin{proposition}\label{prop:bouchut}
Given $\{c_\zeta, \zeta=0\dots Q-1\}$ consider % affine ? useful ? !! product -- essentially bounded
Maxwellians % $\bar f \in C^2(\R^L,\R^{Q\times L})$ 
$\bar f_{\zeta}%,l
:\R^L\to\R^L$ \emph{linear}, that satisfy \eqref{req1}--\eqref{req2} and
\mycya{$$ (*) \qquad \sum\limits_{\zeta=0}^Q \max\limits_{\|q\|_\infty=1} |\bar f_{\zeta,l}(q)| % ,l
\le 1 \quad \forall l=1\ldots L\,.$$} 
%  $\Rightarrow |\sum_{\zeta=0}^Q\bar f_{\zeta,l}q_{\zeta}| % ,l
% \le \max_{\zeta} %=0\ldots Q  \Ccal
% |q_{\zeta}| %,l
% \quad \forall q_{\zeta}%,l
% \in\RR$. %  for some $\Ccal\in(0,1]$
\\
Then, for any smooth solution
\begin{equation}
\label{regularity}
q \in C^0_t\left([0,T),[H^2(\RR^d)]^L\right)
\cap C^1_t\left([0,T),[H^1(\RR^d)]^L\right) \cap C^2_t\left([0,T),[H^0(\RR^d)]^L\right) % any $T>0$
\end{equation}
to \emph{linear} conservations laws \eqref{eq:scl} % associated
with initial condition $q(t=0)=q^0 \in [H^2(\RR^d)]^L$, % s>\frac{d}2+3, & T small if nonlinear
% assuming \eqref{req1} and \eqref{req2}, 
there exists $C_T>0$ % independent of discretization parameters $k$
such that it holds $\forall k=T/N$, $N\in\mathbb{N}^*$: % , $k \max_a \|A_a\|_\infty\le1$
\begin{equation}\label{eq:cvrate1}
\|q(T)-M_0f^N\|_0:=\sqrt{\int_{\R^d}d\boldsymbol{a} {\sum_{l=1}^L} |q_l(T,\boldsymbol{a})-M_0 f^N_l(\boldsymbol{a})|^2} \le k C_T % \quad \forall k=T/N\,, N\in\mathbb{N}^*
\end{equation}
when, for each lattice $\{\boldsymbol{a}_{ij}\}$, $f_{\zeta,l}^n(\boldsymbol{a}_{ij})$ is computed by % the scheme
\eqref{eq:Vbgknp1} with $\omega=1$, and
\begin{equation}
\label{initLBM}
f^0 %_{\zeta,l}
= \bar f( q^0 ) \,.
\end{equation}
% \begin{equation}\label{eq:cvrate1lattice}
% \sqrt{\sum_{ij}|q(T,a_{ij})-M_0 f^N(a_{ij})|^2} \le k C \quad \forall k=T/N\,, N\in\mathbb{N}^*
% \end{equation}
Denoting $F^l_a(q)=\sum\limits_{m=1}^L A^{lm}_aq_{m}$ the flux, % in~\eqref{req2} 
the constant $C_T$ can be chosen as follows: % depends precisely on $T$
\begin{equation*}\label{CT}
C_T = 2 T \left( \max_\zeta|c_\zeta| + \max_a\|A_a\|_\infty \right)^2 \max_{a,b} \sup_{s\in(0,T)} \|\partial^2_{ab}q\|_0(s) \,.
\end{equation*}
\end{proposition}
\begin{proof}
First, % note that
for $\omega=1$ the ``transport-projection'' scheme \eqref{eq:Vbgknp1} rewrites
\begin{equation}
\label{eq:Vbgknp1bis}
f_{\zeta,l}^{n+1}(\boldsymbol{a}) = \bar f_{\zeta,l}(M_0\{f^n_{\zeta}(\boldsymbol{a}-kc_\zeta)\}) % \quad l=1\ldots L
\end{equation}
whatever $k>0$.
Then, since $M_0f^0=q^0\in H^2(\RR^d)$, one can establish by recurrence for $n=1\ldots N$
that $M_0f^n = \sum_\zeta \bar f_{\zeta}(M_0\{f^{n-1}_{\zeta}(\cdot-kc_\zeta)\})\in H^2(\RR^d)$,
and $\max_{a,b}\|\partial^2_{ab}M_0 f^n\|_0 \le % \Ccal^n 
\max_{a,b}\|\partial^2_{ab}q^0\|_0$
by assumptions on $\bar f_{\zeta}$. % linear ``bounded'' operators on fnite-dim space $\RR^L$

Moreover, $f_{\zeta,l}^{n+1}(\boldsymbol{a}')$ % = f_{\zeta,l}^{n+\frac12}(a-kc_\zeta))
in  % that two-step ``transport-projection'' scheme \eqref{eq:Vbgknp1bis}, the transport step
\eqref{transportStep} is % can be computed as
equivalently % the value
$f_{\zeta,l}(t^{n+1},\boldsymbol{a}')$ % of
the solution to
\begin{equation}
\label{transportStepEDP}
(\partial_t + c_\zeta\cdot\nabla)f_{\zeta,l} = 0 \;\ \forall t\in(t^n,t^{n+1}), \boldsymbol{a}\in\R^d
\quad \,; f_{\zeta,l}(t^n,\boldsymbol{a})=f_{\zeta,l}^{n+\frac12}(\boldsymbol{a})
\end{equation}
which possesses the regularity~\eqref{regularity} because $f_{\zeta,l}^{n+\frac12} = \bar f_{\zeta,l}(M_0f^n)\in H^2(\RR^d)$.
Then, using \eqref{transportStepEDP} to Taylor expand $f_{\zeta,l}^{n+1}(\boldsymbol{a}) \equiv f_{\zeta,l}(t^n+k,\boldsymbol{a})$ %t\in [t^n,t^{n+1}]
implies whatever $\boldsymbol{a}\in\R^d$
\begin{multline}
\label{transportStepTaylor}
f_{\zeta,l}^{n+1}(\boldsymbol{a}) = f_{\zeta,l}^{n+\frac12}(\boldsymbol{a}) % f_{\zeta,l}(t^n, \boldsymbol{a})
- k (c_\zeta\cdot\nabla)f_{\zeta,l}^{n+\frac12}(\boldsymbol{a})
% \\
+ \int_{t^n}^{t^{n+1}}(t^{n+1}-s)\partial^2_{tt}f_{\zeta,l}(s, \boldsymbol{a})ds
%
% + \frac{k^2}2 (c_\zeta\cdot\nabla)(c_\zeta\cdot\nabla)f_{\zeta,l}^{n+\frac12}(a)
% + \int_{t^n}^{t^{n+1}}(t^{n+1}-s)^2\partial^3_{ttt}f_{\zeta,l}(s, \boldsymbol{a})ds % / 2 ??
\end{multline}
thus, using \eqref{req1}--\eqref{req2} 
and recalling the projection step $f_{\zeta,l}^{n+\frac12}(\boldsymbol{a})=\bar f_{\zeta,l}\left(M_0f^n(\boldsymbol{a})\right)$,
\begin{multline}
\label{transportStepTaylorM0}
M_0 f_l^{n+1}(\boldsymbol{a}) = M_0 f_l^n(\boldsymbol{a}) % f_{\zeta,l}(t^n, \boldsymbol{a})
- k \myred{\sum_{a=1}^d} \partial_a F_a^l\left( M_0 f^n(\boldsymbol{a}) \right)
\\
+ \int_{t^n}^{t^{n+1}}(t^{n+1}-s)\partial^2_{tt} M_0 f_l(s, \boldsymbol{a})ds
% \\
% + \frac{k^2}2 \partial_a \partial_\beta M_{\alpha\beta} f_l^n(\boldsymbol{a})
% + \int_{t^n}^{t^{n+1}}\frac12(t^{n+1}-s)^2\partial^3_{ttt} M_0 f_l(s, \boldsymbol{a})ds
\,.
\end{multline}
% where we denote $M_{\alpha\beta} f_l := % \sum\limits_{\zeta=0}^{Q-1}
% c_\zeta^\alpha c_\zeta^\beta f_{\zeta,l}$ following \eqref{eq:Ma}. 
Eq.~\eqref{transportStepTaylorM0} can be compared with a Taylor expansion of $q(nk+k,\boldsymbol{a})$ % $q(t=0)\in [H^s(\RR^d)]^L$--$s>1$,%2+\frac{d}2
using \eqref{eq:scl} i.e.
\begin{multline}
\label{sclTaylor}
q_l((n+1)k,\boldsymbol{a}) = q_l(nk,\boldsymbol{a}) - k \myred{\sum_{a=1}^d} \partial_a F_a^l\left(q(nk,\boldsymbol{a})\right)
\\
+ \int_{t^n}^{t^{n+1}}(t^{n+1}-s)\partial^2_{tt} q_l(s, \boldsymbol{a})ds
% \\
% + \frac{k^2}2\partial_a\left(\partial_{m}f^l_\alpha\partial_{m'}f^{m}_\beta\partial_\beta q_{m'}\right)(nk,\boldsymbol{a})
% + \int_{t^n}^{t^{n+1}}\frac12(t^{n+1}-s)^2\partial^3_{ttt} q_l(s, \boldsymbol{a})ds
\end{multline}
where % $\partial_{m}f^l_\alpha\partial_{m'}f^{m}_\beta=A^{lm}_\alpha A^{mm'}_\beta$, so
$\partial^2_{tt} q_l = \myred{\sum_{a,b=1}^d \sum_{m,m'=1}^L} A^{lm}_a A^{mm'}_b \partial^2_{ab} q_{m'}$
% $\partial_a\left(\partial_{m}F^l_a\partial_{m'}F^{m}_b\partial_b q_{m'}\right)$
% i.e. $A^{lm}_a A^{mm'}_b \partial^2_{ab} q_{m'}$ (recall $\partial_{m}F^l_a=A^{lm}_a$)
here in the linear flux case.

Subtracting \eqref{sclTaylor} from \eqref{transportStepTaylorM0}, moreover using \eqref{transportStepEDP} we obtain
\begin{multline}
\label{eq:diffTaylor} 
q_l((n+1)k,\boldsymbol{a}) - M_0 f_l^{n+1}(\boldsymbol{a})
\\
= q_l(nk,\boldsymbol{a}) - M_0 f_l^n(\boldsymbol{a})
- k \myred{\sum_{a=1}^d \sum_{m=1}^L}
\partial_{m}F^l_a \left( \partial_a q_{m}(nk,\boldsymbol{a}) - \partial_a M_0 f_{m}^n(\boldsymbol{a}) \right) + R^n_l(\boldsymbol{a})
\end{multline}
with a remainder $R^n_l(\boldsymbol{a})$ that satisfies by Minkowski's inequality
% \begin{multline}
% \label{Rna}
% \|R^n_l\|_0 % :=\sqrt{\int_{\R^d}da |R^n(\boldsymbol{a})|^2} \\
% \le k^2 \max_{a,b} \Big( \max_\zeta|c_\zeta|^2 \|\partial^2_{ab}M_0 f_l^n\|_0
% + \max_a\|A_a\|_\infty^2 \sup_{s\in(t^n,t^{n+1})} \|\partial^2_{ab}q_l\|_0(s) \Big)
% \end{multline} for all $l=1,\ldots,L$.
\begin{multline}
\label{Rna}
\|R^n\|_0 % :=\sqrt{\int_{\R^d}da |R^n(\boldsymbol{a})|^2} \\
\le k^2 \max_{a,b} \Big( \max_\zeta|c_\zeta|^2 \|\partial^2_{ab}M_0 f^n\|_0
+ \max_a\|A_a\|_\infty^2 \sup_{s\in(t^n,t^{n+1})} \|\partial^2_{ab}q\|_0(s) \Big) \,.
\end{multline}
Using also Minkowski's inequality with \eqref{eq:diffTaylor} yields
% \begin{multline}
% \label{eq:diffTaylor2} 
% \|q_l((n+1)k,\cdot) - M_0 f_l^{n+1}(\cdot)\|_0 % \sqrt{ \sum_{ij} |q_l((n+1)k,\boldsymbol{a}_{ij}) - M_0 f_l^{n+1}(\boldsymbol{a}_{ij})|^2 }
% \le 
% \|q_l(nk,\cdot) - M_0 f_l^{n}(\cdot)\|_0 % \sqrt{ \sum_{ij} |q_l(nk,\boldsymbol{a}_{ij}) - M_0 f_l^n(\boldsymbol{a}_{ij})|^2 }
% \\
% + k \myred{\sum_{a=1}^d \sum_{m=1}^L}
% |\partial_{m}F^l_a| \|\partial_a q_{m}(nk,\cdot) - \partial_a M_0 f_{m}^{n}(\cdot)\|_0
% + \|R^n_l\|_0. % \sqrt{ \sum_{ij} |R_{ij}|^2 } \,. % \sqrt{ \sum_{ij}|\partial_a q_l(nk,\boldsymbol{a}_{ij})-\partial_a M_0 f_l^n(\boldsymbol{a}_{ij}) |^2 }
% \end{multline}
\begin{multline}
\label{eq:diffTaylor2} 
\|q((n+1)k,\cdot) - M_0 f^{n+1}(\cdot)\|_0 \le \|q(nk,\cdot) - M_0 f^{n}(\cdot)\|_0 
\\
+ k \max_{a,l,m} |\partial_{m}F^l_a| \myred{\sum_{a=1}^d} \|\partial_a q(nk,\cdot) - \partial_a M_0 f^{n}(\cdot)\|_0 + \|R^n\|_0. 
\end{multline}
On the other hand, the following Taylor expansions
\begin{eqnarray}
\label{transportStep0TaylorM0}
M_0 f_l^{n+1}(\boldsymbol{a}) & = M_0 f_l^n(\boldsymbol{a}) + \int_{t^n}^{t^{n+1}}\partial_{t} M_0 f_l(s, \boldsymbol{a})ds
\\ 
\label{scl0Taylor}
q_l((n+1)k,\boldsymbol{a}) & = q_l(nk,\boldsymbol{a}) + \int_{t^n}^{t^{n+1}}\partial_{t} q_l(s, \boldsymbol{a})ds
\end{eqnarray}
yield, using differentiation % derivation ? of \eqref{transportStep0TaylorM0} and \eqref{scl0Taylor}
by $\partial_a$, subtraction and Minkowski's inequality
\begin{multline}
\label{eq:diffTaylor1} 
\|\partial_a q(nk,\cdot) - \partial_a M_0 f^n(\cdot)\|_0 \le \|\partial_a q((n-1)k,\cdot) - \partial_a M_0 f^{n-1}(\cdot)\|_0
\\ + k \max_b \left( \max_\zeta|c_\zeta| \|\partial^2_{ab}M_0 f^{n-1}\|_0
+ \max_\alpha\|A_\alpha\|_\infty \sup_{s\in(t^{n-1},t^n)} \|\partial^2_{ab}q\|_0(s) \right) \,.
\end{multline}
% \begin{multline}
% \label{eq:diffTaylor1} 
% \|\partial_a q_l(nk,\cdot) - \partial_a M_0 f_l^n(\cdot)\|_0
% % \sqrt{ \sum_{ij} |\partial_a q_l((n+1)k,\boldsymbol{a}_{ij}) -\partial_a M_0 f_l^{n+1}(\boldsymbol{a}_{ij}) |^2 }
% \le 
% \|\partial_a q_l((n-1)k,\cdot) - \partial_a M_0 f_l^{n-1}(\cdot)\|_0
% % \sqrt{ \sum_{ij} |\partial_a q_l(nk,\boldsymbol{a}_{ij}) - \partial_a M_0 f_l^n(\boldsymbol{a}_{ij})|^2 }
% \\
% + k \max_b \left( \max_\zeta|c_\zeta| \|\partial^2_{ab}M_0 f_l^{n-1}\|_0
% + \max_\alpha\|A_\alpha\|_\infty \sup_{s\in(t^{n-1},t^n)} \|\partial^2_{ab}q_l\|_0(s) \right) \,.
% \end{multline}
Using \eqref{eq:diffTaylor1} in \eqref{eq:diffTaylor2} and iterating for $n=N-1\ldots 0$ yields
% , when $k \max_a \|A_a\|_\infty \le 1$, % not to amplify the remainder in the gradient term ??
\begin{multline}
\label{eq:TaylorBound} 
\|q(T,\cdot) - M_0 f^N(\cdot)\|_0 \le \|q(0,\cdot) - M_0 f^0(\cdot)\|_0
\\
+ (kN\equiv T) \max_a\|A_a\|_\infty \myred{\sum_{a=1}^d} \|\partial_a q(0,\cdot) - \partial_a M_0 f^0(\cdot)\|_0 + C_T k 
\end{multline}
% \begin{multline}
% \label{eq:TaylorBound} 
% \|q_l(T,\cdot) - M_0 f_l^N(\cdot)\|_0
% \le 
% \|q_l(0,\cdot) - M_0 f_l^0(\cdot)\|_0
% \\
% + (kN\equiv T) \max_a\|A_a\|_\infty \max_{a,m} \|\partial_a q_{m}(0,\cdot) - \partial_a M_0 f_{m}^0(\cdot)\|_0
% % \sqrt{ \sum_{ij} |\partial_a q_l(0,\boldsymbol{a}_{ij}) - \partial_a M_0 f_l^0(\boldsymbol{a}_{ij})|^2 }
% + C_T k % T \left( \max_\zeta|c_\zeta| \sup_{s\in(0,T)} \|\partial^2_{\alpha\beta}M_0 f_l^n\|_0(s) + \max_a\|A_a\|_\infty \sup_{s\in(0,T)} \|\partial^2_{\alpha\beta}q_l\|_0(s) \right) ++++
% \end{multline}
that is \eqref{eq:cvrate1}, since \eqref{initLBM} and \eqref{req1} imply $M_0f^0=q^0$.
\end{proof}

\mycya{The proof above shows the consistency of kinetic schemes under the assumption $(*)$ to ensure the ``luxurious'' $H^s$ stability ($s=2$) of the schemes (namely $\max_{a,b}\|\partial^2_{ab}M_0 f^n\|_0 \le % \Ccal^n
\max_{a,b}\|\partial^2_{ab}q^0\|_0$).
% (linear) transport step \eqref{transportStep} in \eqref{eq:diffTaylor}, plus projection \eqref{projectionStep}
That condition $(*)$ %, which ensures $H^s$ stability in Prop.~\ref{prop:bouchut} ($s=2$),
is \emph{too strong}: combined with~\eqref{req1} it leads to $\bar f_{\zeta,l}(q)=\sum_{m}\Omega_{\zeta,m}^l q_m$
with $\sum_{\zeta}\sum_{m}|\Omega_{\zeta,m}^l| \in(0,1)$ for all $l$ which strongly limits the fluxes captured by~\eqref{req2}. 
That is why we do not claim here to have proved any (interesting) convergence result about kinetic scheme
and we do not proceed to a (cumbersome) extension of Prop.~\ref{prop:bouchut} to the case $\omega=2$.}
% \footnote{ %%%%%%%%%%%%%%%%%%%%%%%%%% technically more involved
% Our proof can be adapted for \emph{smooth} solutions to nonlinear systems on adding in RHS of \eqref{eq:diffTaylor} below
% $- k \left( \partial_{m}f^l_a \left(q_l(nk,\boldsymbol{a})\right) -  \partial_{m}f^l_a \left(M_0 f_l^n(\boldsymbol{a})\right) \right)
% \partial_a q_l(nk,\boldsymbol{a})$, a term which could be handled under e.g. the assumption of Lipschitz fluxes.
% Then, % *** assuming smooth Maxwellians that preserve $H^2$ bound ***,
% one would finally obtain a similar conclusion as \eqref{eq:cvrate1} i.e. convergence at a rate $k$.
% } %%%%%%%%%%%%%%%%%%%%%%%%%%%%%%%%%%%%
%
However, the analysis above can guide one further in the construction of kinetic schemes.
Indeed % our error analyis is not optimal.
note that the approximation of \eqref{scl0Taylor} by \eqref{transportStep0TaylorM0} is in fact one order higher than \eqref{eq:diffTaylor1} (one can Taylor expand one order higher % note regularity !!
 \eqref{scl0Taylor} and \eqref{transportStep0TaylorM0}, next get further cancellation using \eqref{req2}).
Thus, a higher-order approximation than the first-order rate \eqref{eq:cvrate1}
can be reached % on assuming a smoother solution
when $\omega=1$: it suffices that the $\frac{d(d+1)}2$ conditions % on the second-order moments
\begin{equation}\label{req3}
M_{ab} \bar f_l(q) := \sum\limits_{\zeta=0}^{Q-1} c_\zeta^a c_\zeta^b \bar f_{\zeta,l}(q)
= \myred{\sum_{m,m'=1}^L} A^{lm}_a A^{mm'}_b q_{m'} \quad \forall a,b=1\ldots d
\end{equation}
be satisfied so as to cancel the first remaining term in the difference of the % high-order 
Taylor expansions \eqref{sclTaylor} and \eqref{transportStepTaylorM0}. 
% the leading error term is $$ \int_{t^n}^{t^{n+1}}(t^{n+1}-s)\partial^2_{\alpha\beta} \left(M_{\alpha\beta}f^l-  A^{lm}_a A^{mm'}_\beta q_{m'}\right)(s,\boldsymbol{a}) ds $$.
Then one may want to require the $\frac{d(d+1)}2$ conditions in the construction of convergent kinetic schemes.
In fact, satisfying the $\frac{d(d+1)}2$ conditions is impossible in general:
see e.g. the next Section~\ref{sec:application} where a practical application case is investigated.
From the analysis of the main error term, a sensible choice to construct a kinetic scheme is then: minimize
\begin{equation}\label{req3bis}
B_{ab,l}(q):= M_{ab} \bar f_l(q) - \myred{\sum_{m,m'=1}^L} A^{lm}_a A^{mm'}_b q_{m'}
\end{equation}
in some norm on $\RR^{\frac{d(d+1)}2}$. \myblu{Such a requirement for the construction of kinetic schemes seems new to us in the literature, although the matrix $B(q)\in\R^{d\times L,d\times L}$ has already appeared (for instance, in \cite[Prop.~1]{aregbadriollet-natalini-2000}, % (2.8)
the positivity of $B(q)$ is shown a sufficient condition for the diffusivity of the equation formally satisfied by the approximation $M_0f$ of $q$).} % following a so-called Chapman-Enskog expansion
In Section~\ref{sec:numerical}, we next investigate numerically the interest of such a % construction 
requirement.

\smallskip

On the other hand, to construct kinetic schemes, one would also like to ensure % require
sufficient stability conditions for the scheme \eqref{eq:Vbgknp1} to converge.
Now, other stability conditions less strong than $(*)$
can be found in the literature e.g. \cite{bouchut-1999,bouchut-2003,bouchut-2004,Dubois2013,Dubois-ijmpc2014}
which have already proved useful numerically in the case of \emph{nonlinear fluxes} % more difficult
(even though rigorous convergence has not yet been proved to our knowledge),
and which turn out compatible with \eqref{req1} and \eqref{req2}. Precisely,
denoting $\eta^*(\varphi):=\sup_q(\varphi^Tq-\eta(q))$ the Legendre transform of $\eta$, % or Fenchel
a $C^2(\RR^l)$ strictly convex functional of the % so-called
\emph{entropy variable} % well-defined for _coercive entropies_
$\varphi^l(q) :=\partial_{q^l}\eta(q) %\dfrac{\partial\eta}{\partial q^l}
$ % dual to the conservative variable of \eqref{eq:scl}
% that reads $\eta^*(\varphi)=\tfrac12|\varphi|^2$ when $\eta(q)=\tfrac12|q|^2$ and
% $\eta^*(\varphi)=\tfrac12\varphi^TS_0^{-1}\varphi$ when $\eta(q)=\tfrac12q^TS_0q$ % linear fluxes
while $\partial_{q^l}\eta$ is a bijection onto $\RR^L$,
% We recall that the (vector) functionals of symmetric-hyperbolic systems \eqref{eq:scl} write, for $\alpha=1\dots d$, $l=1\dots L$
% \[
% f^l_a = \partial_{\varphi^l} g_a^*
% \]
% see e.g. \cite[Th.~5.2, Chapter I, p.42]{Godlewski-Raviart2021}.
a % the following 
``H-theorem'' can % be found to 
further constrain $\bar f$ \cite{bouchut-1999,bouchut-2003}:
% (a $QL$-dimensional operator constrained only by $(d+1)L$ independent equations in the linear flux case, $Q\ge(d+1)$),
% the seminal works \cite{bouchut-1999,bouchut-2003,bouchut-2004} ($d=1$) and \cite{Dubois2013,Dubois-ijmpc2014} ($d\ge1$)
\begin{proposition}
\label{prop:continuous}
Assume there exist $C^1$ % differentiable % smooth % regular
convex functionals $h^*_\zeta$ of % the entropy variable 
$\varphi$ such that
\begin{equation}
\label{decomposition}
    \sum\limits_{\zeta=0}^{Q-1} h^*_\zeta(\varphi) = \eta^*(\varphi)
    \quad
    \sum\limits_{\zeta=0}^{Q-1} c^a_\zeta h^*_\zeta(\varphi) = G_a^*(\varphi)
\end{equation}
hold, using % "the Legendre-transform" % of the entropy flux
$
G_a^*(\varphi) := \varphi^l F^l_a(q(\varphi)) - G_a(q(\varphi)) \,. 
$
Then \eqref{req1}--\eqref{req2} hold on requiring
\begin{equation}
\label{maxwellian}
 \bar f_{\zeta,l}(q) = \dfrac{\partial h^*_\zeta}{\partial\varphi^l}\left(\varphi(q)\right) \,,
\end{equation}
as well as, on denoting % H-theorem for 
$h_\zeta(f_\zeta) = \sup_{\varphi}\left(\varphi^l f_{\zeta,l}- h^*_\zeta(\varphi)\right)$ % where f=f_{\zeta,l}
for solutions to \eqref{eq:Vbgk}:
\begin{equation}
\label{Htheorem}
\partial_t \left(\sum\limits_{\zeta=0}^{Q-1} h_\zeta(f)\right)
+ \myred{\sum_{a=1}^d}  \partial_a \left(\sum\limits_{\zeta=0}^{Q-1} c^a_\zeta h_\zeta(f)\right) \le 0 \,.
\end{equation}
\end{proposition}
\noindent The proof of Prop.~\ref{prop:continuous} is straightforward by direct computations, follow e.g. \cite{Dubois-ijmpc2014} (in the present smooth case). The key point in \eqref{Htheorem} is that whatever $\zeta$, for all $l$ it holds
$\partial_{f_{\zeta,l}}h_\zeta(f)=\varphi^l$ for all $\zeta$, with $\varphi$ such that $f_{\zeta,l}=\partial_{\varphi_l}h^*_\zeta(\varphi)$,
in particular $\partial_{f_{\zeta,l}}h_\zeta(\bar f(q))=\varphi^l(q)\equiv\partial_{q^l}\eta(q)$ whatever $q$.

The importance of requiring (a discrete equivalent to) Prop.~\ref{prop:continuous} % in the construction of kinetic schemes
for numerical stability does not seem well established yet. But it can quite easily enforced here, see Section~\ref{sec:application}, and we shall therefore consider it here for the practical construction of kinetic schemes.
% with a recipe that ensures % some numerical
% stability beyond the linear-flux case.
Precisely, for % specific application to
linear symmetric-hyperbolic systems of conservation laws \eqref{eq:scl} % like elastodynamics
such that $q\in\RR^L$ are actually the variables after symmetrization (i.e. $F_a^l(q)=A_a^{lm}q_m$ with $A_a^{lm} = A_a^{ml}$ for all $a=1\ldots d$, $1\le m,l\le L$) in particular, we require \emph{linear} Maxwellians $\bar f_{\zeta,l}(q)=\myred{\sum_{m=1}^L}  \Omega_{\zeta,l}^{m}q_{m}$
that satisfy \eqref{req1} and \eqref{req2} with \emph{symmetric positive matrices} % necessary ??
$\Omega_\zeta\in\R^{L\times L}$ for all $\zeta=0\ldots Q-1$, % to ensure \eqref{Htheorem} discrete
further denoting $[\Omega_{\zeta}]_{lm}=\Omega_{\zeta,l}^{m}$ their entries. On noting \eqref{maxwellian}, this is equivalent to require $h_\zeta(f_\zeta)=\frac12 \myred{\sum_{l,m=1}^L} f_{\zeta,l}[\Omega_{\zeta}^{-1}]_{lm}f_{\zeta,m}$ where $\Omega_{\zeta}^{-1}$ denotes the matrix inverse of $\Omega_{\zeta}$ (standard inverse if $\Omega_{\zeta}$ is strictly positive equiv. positive definite, or generalized % inverse restriction to the (component) subspace where the linear operator $\Omega_{\zeta}$ is injective
well-defined by e.g. SVD in the semi-definite case). Then indeed, the following discrete version of Prop.~\ref{prop:continuous} holds:
\begin{proposition}\label{prop:discrete}
Assume $h_\zeta(f_\zeta)=\frac12 f_{\zeta,l}[\Omega_{\zeta}^{-1}]_{lm}f_{\zeta,m}$ 
% where $\Omega_{\zeta}^{-1}$ denotes the inverse of symmetric strictly positive matrices $\Omega_{\zeta}$ acting on (possibly a subspace of) $\R^d\ni f_{\zeta}$ % already said !!!!
then it holds for the kinetic scheme \eqref{eq:Vbgknp1} if $\omega=1$, or with equality if $\omega=2$:
$$ 
\sum\limits_{\zeta=0}^{Q-1} h_\zeta(f_\zeta^{n+1}) + \int_{t^n}^{t^{n+1}} \myred{\sum_{a=1}^d} \partial_a \left( \sum\limits_{\zeta=0}^{Q-1} c^a_\zeta h_\zeta(f_\zeta) \right) ds \le \sum\limits_{\zeta=0}^{Q-1} h_\zeta(f_\zeta^{n})
$$
where $f_\zeta$ is the solution to the transport step of $f^{n+\frac12}$ to $f^{n+1}$. % with an equality when $\omega=2$.
\end{proposition}
\begin{proof}
The transport step of $f^{n+\frac12}$ to $f^{n+1}$ satisfies
$$
\partial_t h_\zeta(f_\zeta) + \myred{\sum_{a=1}^d} \partial_a \left( c^a_\zeta h_\zeta(f_\zeta) \right) = 0
$$
thus % pointwise for all a
$$ 
h_\zeta(f_\zeta^{n+1}) + \int_{t^n}^{t^{n+1}} \myred{\sum_{a=1}^d} \partial_a \left( c^a_\zeta h_\zeta(f_\zeta) \right) ds = h_\zeta(f_\zeta^{n+\frac12})
$$
while, using the convexity of $h_\zeta$ and $f_{\zeta,l}^{n+\frac12} =
\myred{\sum_{m'=1}^L} [\Omega_{\zeta}]_{lm'}M_0f_{m'}^n $, it holds:
\begin{multline} \label{32}
h_\zeta(f_\zeta^{n+\frac12}) - h_\zeta(f_\zeta^{n}) \le \partial_{f_\zeta}h_\zeta(f_\zeta^{n+\frac12})\cdot(f_\zeta^{n+\frac12}-f_\zeta^{n})
\\
= % \frac12 % correction
\myred{\sum_{m=1}^L}
f_{\zeta,l}^{n+\frac12}[\Omega_{\zeta}^{-1}]_{lm}
\left( \myred{\sum_{m'=1}^L} [\Omega_{\zeta}]_{mm'}M_0f_{m'}^n - f_{\zeta,m}^n \right) % \omega
= % \frac12 % correction
\myred{\sum_{m=1}^L}
M_0f_{m}^n \left( \myred{\sum_{m'=1}^L} [\Omega_{\zeta}]_{mm'}M_0f_{m'}^n - f_{\zeta,m}^n \right) 
\end{multline}
where the last term in \eqref{32} satisfies 
\begin{equation} \label{33}
\sum_\zeta \left( \myred{\sum_{m'=1}^L} [\Omega_{\zeta}]_{mm'}M_0f_{m'}^n - f_{\zeta,m}^n \right) = 0 
\end{equation}
which concludes the case $\omega=1$. On the other hand, using $x^2-y^2=(x+y)(x-y)$ and $f_{\zeta,l}^{n+\frac12} + f_{\zeta,l}^{n} = 2\myred{\sum_{m'=1}^L}[\Omega_{\zeta}]_{lm'}M_0f_{m'}^n $, it holds when $\omega=2$:
\begin{multline} \label{34}
h_\zeta(f_\zeta^{n+\frac12}) - h_\zeta(f_\zeta^{n}) % = \sum_{\zeta,m,m'} (f_{\zeta,m}^{n+\frac12} + f_{\zeta,m}^{n})[\Omega_{\zeta}]_{mm'}(f_{\zeta,m'}^{n+\frac12} + f_{\zeta,m'}^{n})
= \frac12\Omega_{\zeta}^{-1}(f_{\zeta,l}^{n+\frac12} + f_{\zeta,l}^{n})\cdot(f_\zeta^{n+\frac12}-f_\zeta^{n}) 
\\
= \myred{\sum_{m=1}^L}
M_0f_{m}^n \left( \myred{\sum_{m'=1}^L} [\Omega_{\zeta}]_{mm'}M_0f_{m'}^n - f_{\zeta,m}^n \right)
\end{multline}
and one again finishes the proof on noting that the last term in \eqref{34} satisfies \eqref{33}.
\end{proof}

\mypin{A similar entropy stability as in Prop.~\ref{prop:discrete} is proved in \cite[Chapter 6]{HELIE_Romane_2023_ED269} for the D1Q2 kinetic scheme applied to the scalar transport equation with $\omega\in(1,2)$.}
% Note that other stability conditions are also investigated in \cite[Chapter 6]{HELIE_Romane_2023_ED269} 

\section{Construction of kinetic schemes for 2D elastodynamics}
\label{sec:application}

The $2$-dimensional linear (Hookean) elastodynamics system reads %\footnote{The notation $F_a^m$ for the displacement gradient used below until section~\ref{sec:D2Q5} should not be confused with that generically used above in the previous sections for the fluxes $F^l_a$, e.g. in \eqref{eq:scl}.} % and elsewhere
\begin{equation} 
\label{sys:elasto}
\begin{cases}
\partial_t \FF_a^m - \partial_a u^m = 0 \\ % \partial_t F_a^{x,y}-\partial_a u^{x,y}=0,\\ \partial_t F_b^{x,y}-\partial_b u^{x,y}=0,\\ 
\partial_t u^m - c^2 \partial_a \FF^m_a =0 %  \left( \partial_a F_a^{x,y} - \partial_b F_b^{x,y} \right) 
\end{cases} \,.
\end{equation}

It governs % the Lagrangian description of 
the position $\sum_{m=1}^2x_m(t,\sum_{a=1}^2\boldsymbol{a}_a E^a)e^m$ of a % non-compact
$2$-dimensional body, in the Euclidean physical space equipped with a Cartesian coordinate system $\{e^1,e^2\}$,
using Lagrangian coordinates in a Cartesian system $\{E^1,E^2\}$ equipping a reference configuration of the body,
assuming $\FF_a^m=\partial_a x^m$, $u^m=\partial_t x^m$.

It is a paradigmatic \emph{Friedrichs symmetrizable % symmetric-hyperbolic
linear system of conservation laws}:
$$
q=(\sqrt{c}\FF^1_1,\sqrt{c}\FF^1_2,u^1/\sqrt{c},\sqrt{c}\FF^2_1,\sqrt{c}\FF^2_2,u^2/\sqrt{c})
$$
is solution to a symmetric-hyperbolic system of conservation laws. % with fluxes $F_a=(-{c}u^1,0,-c{c} \FF^1_a,-{c}u^2,0,-c{c} \FF^2_a)$ for $a=1,2$. % and $F_2=(0,-{c}u^1,-c{c} \FF^1_2,0,-{c}u^2,-c{c} \FF^2_2)$.
% The Helmholtz free-energy % of isothermal elastodynamics
% solves an % is the ``mathematical entropy'' solution to the
% aditional conservation law i.e.
% % \begin{equation}
% % \label{eq:energy}
% %  \partial_t \eta + \partial_a \left( c^2 F^m_a u^m \right) = 0 % c^2 f^m u^m
% % \end{equation}
% \begin{equation}
% \label{def:functional}
% \eta %(F_a^m,u^m)
% = \frac12 \sum_{m=1}^2 \left( \sum_a c^2|F^m_a|^2 + |u^m|^2 \right)
% \end{equation}
% % So solutions $q\in C^0\left([0,T),H^s(\RR^2)\right)$ to \eqref{sys:elasto}
% % are well-defined whatever $T>0$ given % six
% % initial fields in $H^s(\RR^2)$, $s\in\RR$ \cite{benzonigavage-serre-2007},
% % \eqref{sys:elasto} is a particular linear case of \eqref{eq:scl} when $d=2$, $L=6$
% so $\eta^*=\eta$, $g_1=-c^2 (u^1 F^1_1+u^2F^2_1)=g_1^*$, $g_2=-c^2 (u^1 F^1_2+u^2F^2_2)=g_2^*$.
%
Noticeably, the linear homogeneous % indifferent to translation in ``body'' (material) space
system \eqref{sys:elasto}, 
%isotropic % indifferent to rotation in ``body'' (material) space
%with uniform $c^2>0$,
symmetrized in $q$, can be interpreted as % also equivalent to 
a vectorial version of the (symmetrized) 2D linear acoustics system for $r=(v,w,p)\in\RR^3$
\begin{equation}
\label{eq:acoustic}
\partial_t r + A_1 \partial_1 r + A_2 \partial_2 r = 0 \\
\end{equation}
that governs velocities $\myred{\sum_{a=1}^2}\sqrt{c} F^m_a E^a \sim v E^1 + w E^2$ and pressures $u^m/\sqrt{c} \sim p$ in two fluids $m=1,2$ with (same) sound celerity $c>0$ parametrizing % linear wave-type fluxes
$$
A_1 = \left(\begin{array}{ccc} 0 & 0 & -c \\ 0 & 0 & 0 \\ -c & 0 & 0 \end{array}\right)
\quad
A_2 = \left(\begin{array}{ccc} 0 & 0 & 0 \\ 0 & 0 & -c \\ 0 & -c & 0 \end{array}\right)
\,.
$$

\bigskip

Here, we develop our construction methodology of a kinetic scheme (based on elements from the previous sections) specifically for the symmetric system on 2D Cartesian grids with grid-size $\Delta>0$.
We consider either D2Q5 or D2Q9 sets of speeds $c_\zeta$ % with uniform time steps $k>0$
(see details below).
We require \emph{linear} Maxwellians $\bar f_{\zeta,l}(q)=\Omega_{\zeta,l}^{m}q_{m}$
satisfying \eqref{req1}, \eqref{req2} and \eqref{decomposition},
which is natural insofar as $\eta^*,G_1^*,G_2^*$ are quadratic functionals of $\varphi$. % $F^m_a,u^m$,
% Still, many candidate Maxwellians exist and
% we propose to require that the kinetic scheme preserves a number of properties of the target system
Moreover we require that the kinetic scheme preserves the % Euclidean
spatial isotropy of the 2D elastodynamics % in the physical Euclidean space
(i.e. $(F^2_1,F^2_2,u^2)$ % $(-c^2 F^2_1,-c^2 F^2_2,-u^2,c^2 F^1_1,c^2 F^1_2,u^1)$
satisfies exactly the same dynamics as $(F^1_1,F^1_2,u^1)$ % $(c^2 F^1_1,c^2 F^1_2,u^1,c^2 F^2_1,c^2 F^2_2,u^2)$
when initial conditions are the same).
So we look only for decompositions \eqref{decomposition} of the form % with separated potentials and linear block Maxwellians
\begin{equation}\label{eq:symPhys}
% h^\star_{\zeta}(\varphi) =  % \sum_{m=1}^2 h^\star_{m,\zeta}(\varphi)
% h^\star_{1,\zeta}(c F^1_1,c F^1_2,u^1) + h^\star_{1,\zeta}(c F^2_1,c F^2_2,u^2)
\partial_{\varphi^l}h^\star_{1,\zeta}(\varphi) %(q)
= \left(\begin{array}{cc} \Omega_\zeta & 0 \\ 0 & \Omega_\zeta \end{array}\right)q^T
\end{equation}
using \emph{symmetric-positive matrices} $\Omega_{\zeta}\in\R^{3\times 3}$ % $\Omega_{\zeta,l,p}$, $l,p\in\{1,2,3\}$
% (for the sake of convexity of $h_\zeta^*$)
which obvioulsy ensure:
\begin{proposition}\label{prop:symPhys}
The numerical approximation
$\sum_{\zeta} f_{\zeta,l}$,  $l\in\{1,2,3\}$, to $({c}F^1_1,{c}F^1_2,u^1/{c})$
satisfies the same dynamics as the numerical approximation $\sum_{\zeta} f_{\zeta,l}$, $l\in\{4,5,6\}$, to $({c}F^2_1,{c}F^2_2,u^2/{c})$
when initial conditions are the same, and \eqref{eq:symPhys} holds.
\end{proposition}

We consider classical sets of kinetic speeds: either the standard D2Q5 case:
$$
c_0 = (0,0); \quad c_1 = \lambda(1,0); \quad c_2 = \lambda(0,1); \quad c_3 = \lambda(-1,0); \quad c_4 = \lambda(0,-1)
$$
with $Q=5$ speeds, or the standard D2Q9 case with four additional speeds:
$$
c_5 = \lambda(1,1); \quad c_6 = \lambda(-1,1); \quad c_7 = \lambda(-1,-1); \quad c_8 = \lambda(1,-1)
$$
which both use grids with step size $\Delta=k\lambda$.
Both cases possess symmetries: % invariance under the multiplicative group generated by a $\pi/2$-rotation: % whatever orientation
given one speed $c_{1}$ in the material space $\RR^2$ equipped with the Cartesian basis $\{E^1, E^2\}$, 
$$
 c_{\zeta} = R_{\pi/2}^{\zeta-1} c_{1}
$$
holds for $\zeta\in\{2,3,4\}$, and
$$
 c_{\zeta} = R_{\pi/2}^{\zeta-5} c_{5}
$$
for $\zeta\in\{6,7,8\}$ when $Q=9$ and $c_{5}=c_1+c_2$, denoting
\begin{equation}
\label{eq:pi2rotation}
R_{\pi/2} = \left(\begin{array}{cc} 0 & -1 \\ 1 & 0 \end{array}\right)
\end{equation}
the (matrix representation of) $\pi/2$-rotation of $\RR^2$.
Now, for the numerical simulation on a \emph{Cartesian} grid of homogeneous physics % materials 
characterized by a uniform constant $c^2$, we naturally also require the preservation of the ambiant-space
symmetries when possible at discrete level, e.g. the $\pi/2$-rotation.

Precisely, still denoting $R_{\pi/2}$ the $3\times 3$ matrix with $2\times 2$ upper-left block \eqref{eq:pi2rotation} complemented by identity (and $R_{-\pi/2}$ its inverse), we require that $f_\zeta$ should satisfy exactly the same dynamics as $R_{\pi/2}^{\zeta-1} f_\zeta$ when initial conditions are also rotated:
\begin{equation}
\label{eq:DMH1}
 \Omega_{\zeta} = R_{\pi/2}^{\zeta-1} \Omega_1 R_{-\pi/2}^{\zeta-1}
\end{equation}
for $\zeta\in\{2,3,4\}$ and
\begin{equation}
\label{eq:DMH2}
 \Omega_{\zeta} = R_{\pi/2}^{\zeta-5} \Omega_5 R_{-\pi/2}^{\zeta-5}
\end{equation}
for $\zeta\in\{6,7,8\}$ when $Q=9$. 

\begin{proposition}
The numerical approximations $(\sum_{\zeta} f_{\zeta,l})$ to $({c}F^1_1,{c}F^1_2,u^1/{c})$ (with $l\in\{1,2,3\}$)
and to $({c}F^2_1,{c}F^2_2,u^2/{c})$ (with $l\in\{4,5,6\}$) are exactly the same as $R_{-\pi/2}(\sum_{\zeta} f_{\zeta,l})$
after rotating the initial conditions similarly.
\end{proposition}

% % {\color{red}%%%%%%%%%%%
% % When $Q=9$, could we require more ?
% % % on considering the linearity of the target PDE / physical process ?
% % % on considering the $\pi/4$-rotation--and--$\sqrt2$-dilation injection of $Z^2$ into $Z^2$ ?
% % It seems difficult in so far as $\pi/4$-rotation--and--$\sqrt2$-dilation
% % is not an automorphism % bijective endomorphism
% % of the square lattice $Z^2$ !
% % }%%%%%%%%%%%%%%%%%%%%%%%%
% %
% % {\color{blue} Can we guide more $\Omega$ toward rank-1 matrices ? It seems this will follow behind,
% % with the purpose of minimizing (numerical/spurious) diffusion.}

\subsection{D2Q5 scheme}
\label{sec:D2Q5}

Applying the contruction principles above to approximate linear elastodynamics with D2Q5 speeds,
it remains to specify % the symmetric matrix
$$
\Omega_1 = 
\left(
\begin{array}{ccc}
\alpha & x & y \\
x & \beta & z\\
y & z & \gamma
\end{array}
\right)
$$
insofar as % the matrices
$\Omega_{2,3,4}$ can next be computed from $\Omega_1$ using
% ``isotropy'', but more precisely discrete material homogeneity (DMH)
\eqref{eq:DMH1}:
% because we consider homogeneous materials where waves ought to propagate indifferently in any material direction,
%$$ \Omega_2 = R_{\pi/2} \Omega_1 R_{-\pi/2} \quad \Omega_3 = R_{\pi/2} \Omega_2 R_{-\pi/2} \quad \Omega_4 = R_{\pi/2} \Omega_3 R_{-\pi/2} $$
\begin{equation} % \Omega_4 = R_{-\frac\pi2}\Omega_1R_{\frac\pi2} !!
\label{Omega}
\Omega_2 =
\left(
\begin{array}{ccc}
\beta & -x & -z \\
-x & \alpha & y\\
-z & y & \gamma
\end{array}
\right),\ % \quad
\Omega_3 =
\left(
\begin{array}{ccc}
\alpha & x & -y \\
x & \beta & -z\\
-y & -z & \gamma
\end{array}
\right),\ % \quad
\Omega_4 =
\left(
\begin{array}{ccc}
\beta & -x & z \\
-x & \alpha & -y\\
z & -y & \gamma
\end{array}
\right)
\,.
\end{equation}
% {\color{green}%%%%%%%%%%%%%%%%%%%%%%%%
% $$
% \Omega_2 =
% \left(
% \begin{array}{ccc}
% \beta & -a & -d \\
% -a & \alpha & b\\
% -d & b & \gamma
% \end{array}
% \right), \quad
% \Omega_3 =
% \left(
% \begin{array}{ccc}
% \alpha & a & -b \\
% a & \beta & -d\\
% -b & -d & \gamma
% \end{array}
% \right), \quad
% \Omega_4 =
% \left(
% \begin{array}{ccc}
% \beta & -a & d \\
% -a & \alpha & -b\\
% d & -b & \gamma
% \end{array}
% \right)
% $$
% }%%%%%%%%%%%%%%%%%%%%%%%%
The consistency relations \eqref{decomposition}, equivalent to % and not simply sufficient to imply, as usual
(\ref{req1}--\ref{req2}) with linear Maxwellians, impose % in the case of linear Maxwellians
$$
\lambda\left(\Omega_1-\Omega_3\right) = 
\left(
\begin{array}{ccc}
0 & 0 & -c \\
0 & 0 & 0\\
-c & 0 & 0
\end{array}
\right) 
\quad
\lambda\left(\Omega_2-\Omega_4\right) = 
\left(
\begin{array}{ccc}
0 & 0 & 0 \\
0 & 0 & -c \\
0 & -c & 0
\end{array}
\right) 
$$
i.e.
$$
y=-\frac{c}{2\lambda} \quad z=0 \,,
$$
and % with definition \eqref{def:functional} for entropy
\begin{align}
\Omega_0
% & = \nonumber \left(
% \begin{array}{ccc}
% 1 & 0 & 0 \\
% 0 & 1
%  & 0\\
% 0 & 0 & 1
% \end{array}
% \right)
% - \sum_{\zeta=1}^4 \Omega_\zeta
% \\ &
=
\left(
\begin{array}{ccc}
1 - 2(\alpha+\beta) & 0 & 0 \\
0 & 1 - 2(\alpha+\beta)  & 0\\
0 & 0 & 1 - 4\gamma
\end{array}
\right)
\,.
\end{align}
So there remain infinitely many choices for a scheme like \eqref{eq:Vbgknp1}.
%that converges at first-order when $\omega=1$ following Prop.~\ref{prop:bouchut} \emph{providing it is numerically stable}.Now,

To ensure numerical stability, we moreover require Prop.~\ref{prop:discrete}, that is: % the convexity of the functions $h_\zeta^*$ % for H stability
all the matrices $\Omega_\zeta$ should be positive.
% SYLVESTER
Here, in view of \eqref{Omega} we need only check that $\Omega_1$ and $\Omega_0$ are positive, % semi-definite
i.e. all their principal minors are positive: % not only the leading one
% it is known e.g. from Theorem 15 p. 162 in Linear Algebra and Its Applications De Peter D. Lax
% https://univ-scholarvox-com.extranet.enpc.fr/reader/docid/41000276/page/177
% see also HornJohnson1990 where this is an ``exercise''
% or Meyer2000 section 7.6 Positive Definite Matrices, page 566 (but the proof is ugly)
\begin{multline}
\label{ineq}
\alpha \beta \geq x^2,\ \alpha \gamma \geq \frac{c^2}{4\lambda^2},\ \alpha, \beta, \gamma \geq 0,\
% \alpha \beta \gamma - a^2 \gamma - \frac{c^2\beta}{4\lambda^2} \ge 0,\ \alpha \geq 0,\ redundant using the two above
% while positivity of $\Omega_0$ asks for
\quad
\alpha + \beta \leq 1/2,\ \gamma \leq 1/4 \,.
\end{multline}

\myblu{It is remarkable that here, for a $3\times 3$ SCL in two dimensions, the ``(kinetic) entropy'' conditions \eqref{ineq} imply the positivity of the symmetric matrix $B(q)$ defined in \eqref{req3bis},
which specifically reads in $\RR^6$ here as
$$ % \begin{bmatrix}\end{bmatrix}
\left[\begin{array}{c|c}
M_{11}\bar f- A_1A_1 & -A_1A_2
\\
\hline
-A_2A_1 & M_{22}\bar f- A_2A_2        
      \end{array}\right]\,.
$$
Indeed, our conditions \eqref{ineq} imply
$$
\alpha \gamma \geq \frac{c^2}{4\lambda^2},\frac1{4\gamma}\ge1 \Rightarrow \alpha \geq \frac{c^2}{\lambda^2}
$$
and
$$
\alpha \gamma \geq \frac{c^2}{4\lambda^2},\frac1{2\alpha}\ge1 \Rightarrow \gamma \geq \frac{c^2}{2\lambda^2}
$$
i.e. $B(q)\ge0$. Now, although % not surprising, because already 
established for (discrete velocity) kinetic approximations of scalar equations \cite{aregbadriollet-natalini-2000,HELIE_Romane_2023_ED269}, that implication has not yet been proved generically, for kinetic approximations of any (symmetric-hyperbolic, multidimensional) SCL. Recall that one consequence of the positivity of $B(q)$ is that the Chapman-Enskog expansion of the (discrete velocity) kinetic scheme is actually a diffusive approximation of the target SCL, which is a desired stability requirement, see \cite{aregbadriollet-natalini-2000}. 
}

One noticeable choice then corresponds to optimizing the accuracy of the scheme on aiming at \eqref{req3}
as ``second-order'' consistency conditions (equivalent to minimum numerical diffusion in the standard analysis by equivalent equations), where
\begin{multline*}
M_{11}\bar f\equiv \lambda^2(\Omega_1+\Omega_3)=
2\lambda^2\left(
\begin{array}{ccc}
\alpha & x & 0 \\
x & \beta & 0\\
0 & 0 & \gamma
\end{array}
\right)
\quad
M_{12}\bar f = 0 = M_{21}\bar f
\\
M_{22}\bar f\equiv \lambda^2(\Omega_2+\Omega_4)=
2\lambda^2\left(
\begin{array}{ccc}
\beta & -x & 0 \\
-x & \alpha & 0\\
0 & 0 & \gamma
\end{array}
\right)
\quad
A_1A_1= \left(
\begin{array}{ccc}
c^2 & 0 & 0 \\
0 & 0 & 0\\
0 & 0 & c^2
\end{array}
\right)
\\
A_2A_2= \left(
\begin{array}{ccc}
0 & 0 & 0 \\
0 & c^2 & 0\\
0 & 0 & c^2
\end{array}
\right)
\quad
A_1A_2= \left(
\begin{array}{ccc}
0 & c^2 & 0 \\
0 & 0 & 0\\
0 & 0 & 0
\end{array}
\right)
\quad
A_2A_1= \left(
\begin{array}{ccc}
0 & 0 & 0 \\
c^2 & 0 & 0\\
0 & 0 & 0
\end{array}
\right)\,.
\end{multline*}
The ``second-order'' conditions \eqref{req3} cannot be fully fullfilled, because % D2Q5 imposes
$M_{12}\bar f = 0 = M_{21}\bar f$. But one can minimize numerical diffusion using $\beta=0=x$ and e.g.
\begin{equation*}
% 2\lambda^2\alpha=c^2 \quad 2\lambda^2\gamma=c^2 impossible because \eqref{ineq} : $ \lambda^2/2 \ge c^2\ge \lambda^2$
\underset{\{0\leq \alpha \leq 1/2 \quad 0\leq \gamma \leq 1/4 \quad \lambda^2 \geq \frac{c^2}{4\alpha \gamma}\}}{\mathop{\rm argmin}} J(\alpha,\gamma,\lambda):=\left( |2\lambda^2\alpha-c^2|^2 + |2\lambda^2\gamma-c^2|^2 \right)
% coercive in \gamma,\alpha,\lambda -- convex (strictly) ? Hessian !?
% 2 \lambda^4    -\mu     4\lambda(2\lambda^2\alpha-c^2)
% -\mu       2 \lambda^4  4\lambda(2\lambda^2\gamma-c^2)
%           Y=4(\alpha+\gamma)(\lambda^2(\alpha+\gamma)-2c^2) + 8(\lambda^2\alpha^2+\lambda^2\gamma^2) -3\mu c^2/2\lambda^4
% det(H-XI) = (2\lambda^4-X)(  (2\lambda^4-X)(Y-X) - 4\lambda(2\lambda^2\alpha-c^2)4\lambda(2\lambda^2\gamma-c^2)
% X = 2\lambda^4 ou .5*( Y+2\lambda^4 pm sqrt( delta ) ) >0 dès que
% delta = (Y+2\lambda^4)^2- 4 * 4\lambda(2\lambda^2\alpha-c^2)4\lambda(2\lambda^2\gamma-c^2) >0 CHECK !?
% expand (4*(a+g)*(l^2*(a+g)-2*c^2) + 8*(l^2*a^2+l^2*g^2))^2 -4*4*l*(2*l^2*a-c^2)*4*l*(2*l^2*g-c^2)
% a^4 l^4 - 192 a^3 c^2 l^2 + 192 a^3 g l^4 + 64 a^2 c^4 - 320 a^2 c^2 g l^2 + 352 a^2 g^2 l^4 + 128 a c^4 g - 320 a c^2 g^2 l^2 + 192 a g^3 l^4 + 64 c^4 g^2 - 192 c^2 g^3 l^2 + 144 g^4 l^4
% + 128 a c^2 l^4 - 256 a g l^6 - 64 c^4 l^2 + 128 c^2 g l^4 ??
\end{equation*}
that is,
\begin{equation}\label{choice-min}
\gamma=\frac14 \quad \alpha=\frac12 \quad \lambda=\sqrt2 c     
\end{equation}
on noting $J$ is smooth, coercive, constrained on bounded compact domain,
with extremas reached % using KKT conditions
on boundaries $\alpha=\frac12$, $\gamma=\frac14$ using ``Lagrange multiplyers'' $\mu,\nu$
$$
\gamma \lambda^2 = \frac{c^2}2 + \frac{c^2\mu}{\lambda^2}
\quad
\alpha \lambda^2 = \frac{c^2}2 + \frac{c^2\nu}{\lambda^2}
$$
compatible with $\lambda\ge\sqrt{2}c$.
% solve by Lagrangian J + \mu (positive condition)
In the sequel, we shall compare \eqref{choice-min}
with a more %diffusive choice (same CFL)
%$$ \gamma=\frac14 \quad \alpha=\frac14 \quad \lambda=\sqrt2 c $$
%Attention: ce jeu de paramètres ne vérifie pas la condion $\alpha\gamma \geq c^2/4\lambda^2 = 1/8$
%or 
standard CFL $\frac12$, i.e. $\lambda=2 c$ (so $\alpha\gamma \geq 1/16$), and three extremal values for $\alpha,\gamma$:
\begin{align*}
 & \gamma=\frac14 \quad \alpha=\frac12,  \\ % \quad \lambda=2 c \,. 
  \textrm{or} \quad & \gamma=\frac14 \quad \alpha=\frac14, \\
  \textrm{and also} \quad & \gamma=\frac18 \quad \alpha=\frac12.
  \end{align*}
%***

Note that when $\gamma=\frac14$, $\alpha=\frac12$,
the D2Q5 scheme is in fact a D2Q4 since $\Omega_0=0$.

\subsection{D2Q9 scheme}
\label{sec:D2Q9}

Applying the contruction principles above with D2Q9 speeds, we have only to specify
$$
\Omega_1 =
\left(
\begin{array}{ccc}
\alpha & x & y \\
x & \beta & z\\
y & z & \gamma
\end{array}
\right)
\quad
\Omega_5 =
\left(
\begin{array}{ccc}
\tilde \alpha & \tilde x &\tilde  y \\
\tilde x & \tilde \beta & \tilde z\\
\tilde y & \tilde z & \tilde \gamma
\end{array}
\right)
$$
but not less since it does not seem obvious to link $\Omega_1$ and $\Omega_5$ through symmetries of the Cartesian grid.
The consistency relations \eqref{decomposition} (or \eqref{req1} and \eqref{req2}) impose
\begin{multline}
\lambda\left(\Omega_1-\Omega_3\right) + \sqrt2\lambda\left(\Omega_5-\Omega_6-\Omega_7+\Omega_8\right)=
\left(
\begin{array}{ccc}
0 & 0 & -c \\
0 & 0 & 0\\
-c & 0 & 0
\end{array}
\right)
\\
\lambda\left(\Omega_2-\Omega_4\right) + \sqrt2\lambda\left(\Omega_5+\Omega_6-\Omega_7-\Omega_8\right)=
\left(
\begin{array}{ccc}
0 & 0 & 0 \\
0 & 0 & -c \\
0 & -c & 0
\end{array}
\right)
\end{multline}
i.e.
\begin{align}
2\lambda (y + \sqrt2 \tilde y - \sqrt2 \tilde z) & =-c
\\
2\lambda (z + \sqrt2 \tilde z + \sqrt2 \tilde y) & = 0
\\
2\lambda (y + \sqrt2 \tilde z - \sqrt2 \tilde y) & =-c
\\
2\lambda (z + \sqrt2 \tilde y + \sqrt2 \tilde z) & = 0
\end{align}
so $$ \tilde y = \tilde z = 0 = z \quad y = -\frac{c}{2\lambda} $$
and % with definition \eqref{def:functional} for entropy
\begin{align}
\Omega_0
% & = \nonumber \left(
% \begin{array}{ccc}
% 1 & 0 & 0 \\
% 0 & 1
%  & 0\\
% 0 & 0 & 1
% \end{array}
% \right)
% - \sum_{\zeta=1}^4 \Omega_\zeta
% \\ &
=
\left(
\begin{array}{ccc}
1 - 2(\alpha+\beta+\tilde \alpha+\tilde \beta) & 0 & 0 \\
0 & 1 - 2(\alpha+\beta+\tilde \alpha+\tilde \beta)  & 0\\
0 & 0 & 1 - 4(\gamma+\tilde \gamma)
\end{array}
\right)
\,.
\end{align}
Furthermore, the positivity of $\Omega_0$, $\Omega_1$ and $\Omega_5$ requires:
\begin{multline}
\label{ineq1}
\alpha \beta \geq x^2,\ \alpha \gamma \geq \frac{c^2}{4\lambda^2},\
\tilde \alpha \tilde \beta \geq \tilde x^2,\ \tilde \alpha \tilde \gamma \geq 0,\
\\
\tilde \alpha, \tilde \beta, \tilde \gamma \geq 0,\alpha, \beta, \gamma \geq 0,\
\alpha+\beta+\tilde \alpha+\tilde \beta \leq 1/2,\ \gamma +\tilde \gamma \leq 1/4 \,.
\end{multline}
Finally, on computing
\begin{multline*}
M_{11}\bar f\equiv \lambda^2(\Omega_1+\Omega_3) + 2\lambda^2(\Omega_5+\Omega_6+\Omega_7+\Omega_8) =
2\lambda^2\left(
\begin{array}{ccc}
\alpha + 2(\tilde \alpha+\tilde \beta) & x & 0 \\
x & \beta + 2(\tilde \alpha+\tilde \beta) & 0\\
0 & 0 & \gamma + 2\tilde \gamma
\end{array}
\right)
\\
M_{12}\bar f = M_{21}\bar f \equiv 2\lambda^2(\Omega_5-\Omega_6+\Omega_7-\Omega_8) =
2\lambda^2\left(
\begin{array}{ccc}
2(\tilde \alpha-\tilde \beta) & 4\tilde x & 0 \\
4\tilde x & 2(\tilde \beta-\tilde \alpha) & 0\\
0 & 0 & 0
\end{array}
\right)
\\
M_{22}\bar f\equiv \lambda^2(\Omega_2+\Omega_4)+ 2\lambda^2(\Omega_5+\Omega_6+\Omega_7+\Omega_8)=
2\lambda^2\left(
\begin{array}{ccc}
\beta + 2(\tilde \alpha+\tilde \beta) & -x & 0 \\
-x & \alpha + 2(\tilde \alpha+\tilde \beta) & 0\\
0 & 0 & \gamma + 2\tilde \gamma
\end{array}
\right)
\end{multline*}
minimizing numerical diffusion leads, using \eqref{ineq1}, to
$$
\beta=x=0=\tilde\alpha=\tilde\beta=\tilde\gamma=\tilde x,\ \lambda^2 = 2c^2,\ \gamma=\frac14,\ \alpha=\frac12 %\frac14
$$
that is the D2Q5 scheme already constructed in the previous section~!
Nevertheless, to see the effect of a few additional celerities (i.e. kinetic directions), we could also do the following choice:
\begin{equation}\label{paramd2q9}
\beta = 0 \quad \gamma = \frac14 \quad
\tilde \alpha = \tilde \beta = \tilde x = \frac{c^2}{16\lambda^2} \quad
% an intermedaite solution for cross diffusion which implies
% x \in \left(\frac{c^2}{\lambda^2},\frac12-\frac{c^2}{8\lambda^2}\right) % so \lambda\ge \frac{3c}2
\lambda = \frac{3c}2 \quad \textrm{and} ~ \alpha = \frac49.
\end{equation}
%and $\alpha\in\{\frac49,\frac12\}, \tilde x \in\{0,\frac1{36}\}$.

\section{Numerical evaluation}
\label{sec:numerical}

Here we evaluate numerically our vectorial kinetic schemes built in Section~\ref{sec:application},
for the % numerical
approximation in $L^2\left(\Dcal\right)$, $\Dcal:=\{\|\bolda\|\le R\}$, %\equiv B_R(0) not (-R,R)^2
at $t=T$ of % smooth
solutions to \eqref{sys:elasto} computed on $\Dcal_T \supset \{\|\bolda\|\le R+cT\}$, $t\in(0,T)$. 
% not [-(R+cT),(R+cT)]^2
Note that the conditions used at the boundary $\partial\Dcal_T$ % of $\Dcal_T$ for $t>0$ 
are not influential
insofar as % the simulation domain $\Dcal$ contains the whole influence domain of $(-R,R)^2$
the simulation domain $\Dcal_T$ contains the % initial
dependence domain of $\Dcal$ at $t=T$.
% (outside which initial values can be identified with boundary conditions at $t\in(0,T)$, but do not influence the solution in $(-R,R)^2$).
Simulations can be thus compared on $\Dcal$ % at $t=T$
with smooth analytical % finite-energy
solutions to \eqref{sys:elasto} on $\R^d$, which can be defined e.g. from
\begin{equation}
\label{eq:wave}
  (\partial^2_{tt}-c^2 \myred{\sum_{\alpha=1}^d})u^m=0 \text{ for } m=1,2
\end{equation}
complemented with $u^m|_{t=0}$ and $\partial_t u^m|_{t=0}=c^2\myred{\sum_{\alpha=1}^d}\partial_\alpha F^m_\alpha|_{t=0}$, % analytically 
also computable as % Kirchoff's formula 
\begin{equation}
\label{eq:kirchoff}
 u^m(t,\boldsymbol{a})=\partial_t\frac1{2\pi}\int_{D_{t,\boldsymbol{a}}}\frac{u^m|_{t=0}(\boldsymbol{b})}{\sqrt{c^2t^2-|\boldsymbol{a}-\boldsymbol{b}|^2}}d\boldsymbol{b} 
  - \frac1{2\pi}\int_{D_{t,\boldsymbol{a}}}\frac{\partial_t u^m|_{t=0}(\boldsymbol{b})}{\sqrt{c^2t^2-|\boldsymbol{a}-\boldsymbol{b}|^2}}d\boldsymbol{b}
\end{equation}
where $D_{t,\boldsymbol{a}}:=\{\boldsymbol{b}, ~ c^2t^2\ge|\boldsymbol{a}-\boldsymbol{b}|^2\}$ is the domain of influence on the solution at $(t,\boldsymbol{a})$.
% Moreover,
Focusing here % more specifically 
on axisymmetric % initial conditions, thus also analytical 
solutions (that depend only on $r>0$ the distance to the origin),
% particular solutions are possible using variables separation u^m(t,r)= R_k(r)g_k(t) for any real $k$
% $$ \frac1{r R_k} \partial_r \left( r \partial_r R_k \right) = k^2 = \frac1{c^2g_k} \partial_{tt}^2 g_k $$
% g_k = a_k\sin(wt) + b_k\cos(wt) with c^2w^2 = k^2 and R_k = c_k J_0(kr) + d_k Y_0(kr)
% using Bessel functions of first and second kind J_0,Y_0
% https://mathworld.wolfram.com/HelmholtzDifferentialEquationCircularCylindricalCoordinates.html
let us introduce the (zero-order) Hankel transform of $u^m$ % of order 0 
$$
\widetilde{u^m}(k,t) = \int_0^\infty r J_0(kr)u^m(r,t) dr 
$$
where $J_0$ is zero-order Bessel function, % of first kind 
also often denoted $\mathcal{H}_0(u^m)(k,t)$. % bessel functions are orthogonal !! 10.1088/0143-0807/36/1/015016
It satisfies
$$ \int_0^\infty % r \frac1{r} 
\partial_r \left( r \partial_r u^m(r,t) \right)  J_0(kr) dr = - k^2 \tilde u^m(k,t) $$
so one can also compute solutions to \eqref{eq:wave} by
\begin{equation}
\label{eq:hankel}
u^m(t,r) = \int_0^\infty \left( \widetilde{u^m|_{t=0}}(k) \cos(ckt)   
+ \widetilde{\partial_t u^m|_{t=0}}(k) \frac{\sin(ckt)}c \right) J_0(kr) k dk \,.
\end{equation}
For instance, recalling $r^2 = ||\boldsymbol{a}||^2$, choosing % axisymmetric initial conditions
% \begin{enumerate}
%  \item $u^m|_{t=0}(r)=A(r^2+a^2)^\frac12$ so $\widetilde{u^m|_{t=0}}(k)=\frac{A}k e^{-ak}$
% yields
% \begin{multline*}
% u^m(t,r)=A \mathop{\rm Re}\left(r^2+(a+ict)^2\right)^{\frac12} \text{ CHECK} % TBC
% \\
% \equiv \frac{A}{\sqrt{2}}\left( r^2+a^2-c^2t^2 + \sqrt{ r^4+a^4+c^4t^4 + 2r^2a^2-2r^2c^2t^2 +2a^2c^2t^2 } \right)^{\frac12},
% \end{multline*}
% \item 
$\partial_t u^m|_{t=0}=0=\myred{\sum_{\alpha=1}^d}\partial_\alpha F_\alpha^m|_{t=0}$ 
and
$u^m|_{t=0}(r)=\kappa e^{-\mu r^2}$ yields  \cite[Chap.7]{IntTraAndTheAppThiEdi} $\widetilde{u^m|_{t=0}}(k)=\frac{\kappa}{2 \mu}e^{-\frac{k^2}{4\mu}}$ so
\begin{equation}
\label{eq:anasol}
u^m(t,r)= \frac{\kappa}{2 \mu} \int_0^{+\infty} e^{-\frac{k^2}{4\mu}} \cos(ck t) J_0(k r) k dk = \mathcal{H}_0\left(\frac{\kappa}{2\mu} e^{-\frac{k^2}{4\mu}} \cos(ckt) \right)(t,r). 
\end{equation}
The latter analytic solution can be numerically computed with high precision given $r$,
see e.g. the solution plot in Fig. \ref{fig:analytic_solution} using the open-source {\tt gsl}\footnote{The discrete Hankel transform, with $M=2^{12}$ zeros of $J_0$ and $X=4$ precisely in Fig. \ref{fig:analytic_solution},
using the notations of https://www.gnu.org/software/gsl/doc/html/dht.html}.

\begin{figure}[ht]
    \centering
    %\includegraphics[scale=0.5]{convergence.pdf}
    %\caption{Comparison of $L_2$ norm error on $p$ at final time. Results obtained with parameters $T = \frac14$, $c = 1$, $\lambda = 1$, $\alpha = 400$, $R = \frac14$ and $CFL = 0.5$.}
    %\label{fig:convergence_FD}
    \includegraphics[scale=0.8]{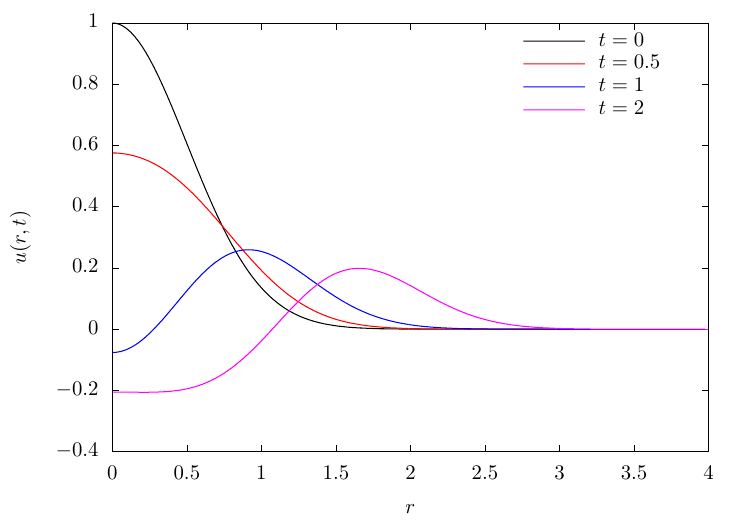}
    \caption{%Analytic 
    Solution \eqref{eq:anasol} for $\kappa = 1, \mu =2$ computed at $t\in\{0;0.5;1;2\}$ as a function of $r$
    using {\tt gsl}.}
    \label{fig:analytic_solution}
\end{figure}
 
% \end{enumerate}
% where we recall $J_0$ is 0-order Bessel function of the first kind

For benchmarking purposes, we also compute numerical approximations % on $\Dcal_T:=[-(R+cT),(R+cT)]^2$, $t\in(0,T)$
using other discretization methods on the same uniform Cartesian grid of $\Dcal_T$, with square cells $\Delta\times\Delta$ centered at $\boldsymbol{a}_{ij}$, $ij\in\Ical:=\{-I,-I+1,\dots,I\}$ ($I\in\NN$),
ensuring a finite propagation speed of the information through \emph{time-explicit} schemes. We impose as initial condition % the exact solution 
$u_0(r)=\kappa e^{-\mu r^2}$ for $\kappa = 1, \mu =2$, with $r = \sqrt{x_i^2+y_j^2}$ at cell centers $(x_i, y_j)\in\Dcal_T:=(-4,4)^2$ in the FV case (i.e. simply grid points in the FD case). % this is a second order approximation of cell-averaged values $u_{ij}^0 = u_0(r) + O(\Delta^2)$. These are detailed in Appendix. % MINH

\smallskip

A mesh convergence study is plotted on Fig. \ref{fig:convergence_FV_BGK} when $R=2$, $c=\frac1{\sqrt{2}}$ at $T=1$.
Our proposed D2Q5 vectorial kinetic schemes are compared one another, when $\omega=1$ and when $\omega=2$,
with % standard 
first and second order FV schemes detailed in Appendix~\ref{sec:FV}. 

\mygreen{First, our numerical experiments show the convergence of the vectorial kinetic schemes,
at rates $O(\Delta)$ and $O(\Delta^2)$ respectively when $\omega=1$ and when $\omega=2$,
whenever the parameters $\lambda, \alpha,\gamma$ satisfy \eqref{ineq}.
% Although the convergence rates were expected see e.g. \cite{Coulette2020} whenever the scheme converges, it is noticeable that second order does converge when $\omega=2$ and not only when $\omega=2^-$, as seen in the literature.
It is remarkable that when only the less stringent conditions $\alpha\ge\frac{c^2}{\lambda^2}$, $\gamma\ge\frac{c^2}{2\lambda^2}$ (so $B(q)\ge0$) are satisfied \emph{strictly} (i.e. $\lambda<c\sqrt{2}$ and \eqref{ineq} are not all satisfied)
% otherwise $\lambda\ge c\sqrt{2}$ + $\alpha\ge\frac{c^2}{\lambda^2}$, $\gamma\ge\frac{c^2}{2\lambda^2}$ => \eqref{ineq}
then we observe numerical instabilities. % blow-up % no convergence
For instance, our D2Q5 schemes all blow up % on $t\in (0,1)$
with % a mesh of $2^{11}$ cells i.e.
$\Delta = 2^{-9} \equiv .001953125$
when $\lambda = c(\sqrt{2} - .1)$ whatever $\omega\in (1,2)$,
and for $\omega \ge 1.9$ when $\lambda = c(\sqrt{2} - .01)$.
Such numerical results are not surprising in view of previous results
\cite{aregbadriollet-natalini-2000,HELIE_Romane_2023_ED269}.
A full rigorous theory regarding the necessary/sufficient conditions for convergence of the kinetic schemes toward SCL however remains needed.}
% As we can see, the FV scheme can find as well the expected first and second order convergence rate.

\mygreen{Second, we observe that the kinetic schemes with largest $CFL = 1/\sqrt{2}$, i.e. using $\lambda = \sqrt{2}c, ~ \gamma = 1/4, ~ \alpha = 1/2$ are the most accurate at both first and second order ($\omega = 1, ~ \omega = 2$).
That parameter choice was motivated by minimizing the numerical diffusion when $\omega = 1$ (recall the proof of Prop.~\ref{prop:bouchut}).
The D2Q9 schemes using a slightly more diffusive CFL coefficient $\lambda = 3c/2$ from \eqref{paramd2q9}
are second best, next come the cases $\lambda = 2c$ corresponding to a 1/2 CFL condition.
Moreover, when simulated with $\omega=1$ or $2$ in the stability region of Prop.~\ref{prop:discrete},
% we have obtained the expected convergence results at first and second order and
different (admissible) values of $\alpha, \gamma$ have remarkable impact on the accuracy of numerical solution.
This also call for further studies (of the discretization error), theoretically and numerically (note that $\beta\neq0$, $x\neq0$ in the Maxwellians has not been tested here.)
}

Interestingly, we observe that the convergences of the first order FV scheme and that of the first order kinetic scheme with $\lambda = 2c, \gamma = 1/4 = \alpha$ are superimposed, which is not surprising insofar as, recalling $\lambda=\frac\Delta{k}$, the scheme 
$$
q^{n+\frac12}_{i,j}
= \Omega_0 q^{n-\frac12}_{i,j}
+ \Omega_1 q^{n-\frac12}_{i-1,j}
+ \Omega_2 q^{n-\frac12}_{i,j-1}
+ \Omega_3 q^{n-\frac12}_{i+1,j}
+ \Omega_4 q^{n-\frac12}_{i,j+1}
$$
that results from LBM for the first-order moments when $\lambda=2c$ i.e. 
$$
\Omega_1 
= \begin{pmatrix} 
\frac14 & 0 & -\frac14 \\
0 & 0 & 0 \\
-\frac14 & 0 & \frac14
\end{pmatrix} 
\quad
\Omega_0 
= \begin{pmatrix} 
\frac14 & 0 & 0 \\
0 & \frac14 & 0 \\
0 & 0 & 0
\end{pmatrix} 
$$
then exactly coincides with the first-order FV scheme at CFL $\frac12$, see (\ref{eq:vhnFV}--\ref{eq:whnFV}--\ref{eq:phnFV}) in Appendix~\ref{sec:FV}. 
Similarly, the second-order % converging 
numerical results for the Yee scheme look very much like those for the kinetic scheme with $\lambda = 2c, \gamma = 1/8, \alpha = 1/2$. \mycomment{though it is not obvious => compute amplification matrix ??
$$
\Omega_1 
= \begin{pmatrix} 
\frac12 & 0 & -\frac14 \\
0 & 0 & -\frac14 \\
\frac-14 & 0 & \frac18
\end{pmatrix} 
\quad
\Omega_0 
= \begin{pmatrix} 
0 & 0 & 0 \\
0 & 0 & 0 \\
0 & 0 & \frac12
\end{pmatrix} 
$$}

% Furthermore, we have numerically observed that the choice of parameters $\lambda, \alpha,\gamma$ respecting condition \eqref{ineq} is mandatory for the stability of numerical solution. These numerical results need more in-deep investigation to confirm the necessity of such stability conditions.

\begin{figure}[ht]
    \centering
    %\includegraphics[scale=0.5]{convergence.pdf}
    %\caption{Comparison of $L_2$ norm error on $p$ at final time. Results obtained with parameters $T = \frac14$, $c = 1$, $\lambda = 1$, $\alpha = 400$, $R = \frac14$ and $CFL = 0.5$.}
    %\label{fig:convergence_FD}
    \includegraphics[scale=0.7]{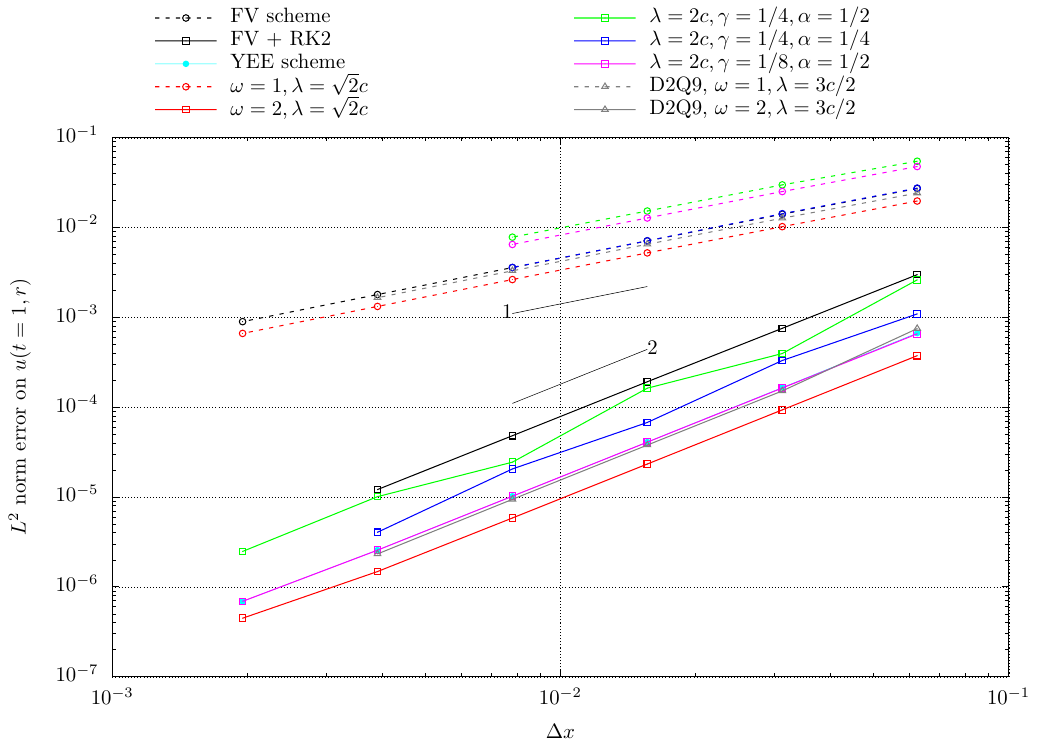}
    \caption{Errors in $L_2(\{\|\bolda\|\le 2\})$ 
    of approximations of \eqref{eq:anasol} % $u^m(T = 1, r)$ 
    at $t=1$ when $\kappa=1$, $\mu=2$, $c=\frac1{\sqrt{2}}$, by (i) various D2Q5, D2Q9 kinetic schemes, 
    with various values of parameters $\alpha,\gamma,\lambda,\omega$,
    % $\alpha\in\{\frac12,\frac14\}$,$\gamma\in\{\frac18,\frac14\}$,$\omega\in\{1,2\}$, 
    and (ii) the % standard 
    FD, FV schemes with CFL = $1/2$ -- on $\Dcal_T=(-4,4)^2$ --.}
    \label{fig:convergence_FV_BGK}
\end{figure}

%\section{Conclusions}\label{sec:conclusions}
%We have proposed a methodology for the construction of (vector) kinetic schemes approximating (linear) symmetric-hyperbolic systems of conservation laws, in the case of linear elastodynamics. The resulting scheme as well at both first and second order, compared with standard FD and FV schemes. In particular, one can define a set of parameters for which the proposed kinetic solver is the most accurate.

%Treatment of boundary conditions ??

\section*{Acknowledgments}
We would like to acknowledge the help of ANR with grant ANR-15-CE01-0013 for the SEDIFLO project. We also thank Lucas Br\'elivet for its help in a preliminary version of our numerical implementation with {\tt basilisk}\footnote{http://basilisk.fr/}.

\appendix

\section{Finite Difference and Finite Volume schemes}

\subsection{Finite-Difference scheme}
\label{sec:FD}

Each ``acoustic'' sub-system $(cF^m_1,cF^m_2,u^m)\equiv(v,w,p)$ for $m=1,2$ can be numerically simulated
using a standard (leap-frog, centered) Finite-Difference scheme%for $(cF^m_1,cF^m_2,u^m)\equiv(v,w,p)$ solution to \eqref{eq:acoustic} on $\Dcal_T$
, often named after Yee \cite{Yee1966}.
%\cite{gustafsson-kreiss-oliger-2013,oreilly2017}
% ** first by ** Yee for electromag (Yee1966) Yee, K. S. Numerical Solution of Initial Boundary Value Problems Involving Maxwell’s Equations in Isotropic Media IEEE Transactions on Antennas and Propagation, 1966, 14, 302-307
% ** next ** J. Virieux, SH-wave propagation in heterogeneous media: Velocity-stress finite-difference method. Geophysics 49 (1984) 1933–1957
% ** Graves1996
Denoting staggered nodal approximations
$$
p_{ij}^n\approx p(t^n,i\Delta,j\Delta)
\quad
w_{ij+\frac12}^{n+\frac12}\approx w(t^n+\frac{k}2,i\Delta,j\Delta+\frac{\Delta}2)
\quad
v_{i+\frac12,j}^{n+\frac12}\approx w(t^n+\frac{k}2,i\Delta+\frac{\Delta}2,j\Delta)
$$
the following Yee scheme writes
\begin{align}
\label{eq:vhn}
v_{i+\frac12,j}^{n+\frac12} & = v_{i+\frac12,j}^{n-\frac12} + \frac{ck}{\Delta} (p_{i+1,j}^n-p_{ij}^n)
\\
\label{eq:whn}
w_{ij+\frac12}^{n+\frac12} & = w_{ij+\frac12}^{n-\frac12} + \frac{ck}{\Delta} (p_{ij+1}^n-p_{ij}^n)
\\
p_{ij}^{n+1} & = p_{ij}^n + \frac{ck}{\Delta} (v_{i+\frac12,j}^{n+\frac12} - v_{i-\frac12,j}^{n+\frac12} + w_{ij+\frac12}^{n+\frac12} - w_{ij-\frac12}^{n+\frac12}) % + g_{ij}^{n+\frac12}
\label{eq:phn}
\end{align}
yields approximations such as the space-continuous reconstructions
$$
\sum_{ij\in\Ical} p_{ij}^n \psi_{ij} \approx p(t^n)
\quad
\sum_{ij\in\Ical} w_{ij+\frac12}^n \psi_{ij+\frac12} \approx w(t^n)
\quad
\sum_{ij\in\Ical} v_{i+\frac12,j}^n \psi_{i+\frac12,j} \approx v(t^n)
$$
using the tensor-product $\QQ_{1,2}$ Finite-Element basis functions $\psi_{ij},\psi_{i+\frac12,j},\psi_{i+\frac12,j}$.
The latter % approximations
converges in $L^2(\Dcal)$ % or discrete FD norm 
to smooth $(v,w,p)\in H^2(\RR^2)$
with an error rate $O(\Delta^2)$ as $\Delta = c k\to0$
\cite{Monk-ConvergenceAnalysisYees-1994}. % possibly non uniform !
% using e.g. periodic boundary conditions on $\Dcal_T$, but we do not care
%
A proof for  $p$ % e.g.
consists in bounding above the RHS of
\begin{multline}
\label{bound}
\|p(t^n)-\sum_{ij\in\Ical} p_{ij}^n \psi_{ij}\|_{L^2(\Dcal)} % ^2
\le 
\|p(t^n)-\sum_{ij\in\Ical}  p(t^n,i\Delta,j\Delta) \psi_{ij}\|_{L^2(\Dcal)} % ^2
\\
+ \sup_{ij\in\Ical} \|\psi_{ij}\|_{L^2(\Dcal)} % ^2 
\sqrt{ \sum_{ij\in\Ical} | p(t^n,i\Delta,j\Delta)-p_{ij}^n |^2 }
\end{multline}
by the sum of an upper-bound for each term:
the first term in RHS of \eqref{bound} can be bounded above by $C_1\Delta^2\|p(t^n)\|_{H^2(\Dcal)}$ 
with a constant $C_1$ function of $p$ % at $t^n$ 
but independent of $\Delta$ using a standard interpolation result % by FE functions -- cite Ern of Brenner Scott ?
% Bramble-Hilbert + scaling for a shape-regular mesh : eg (13.28) in 10.1007/978-3-030-56341-7
and the second term in RHS of \eqref{bound} can be bounded above by $C_2\Delta^2\sup_{t\in[0,t^n]}\|p(t)\|_{H^2(\Dcal)}$ using a standard a priori analysis for the pointwise errors % difference 
like $p(t^n,i\Delta,j\Delta)-p_{ij}^n$ which are solution to (\ref{eq:vhn}--\ref{eq:whn}--\ref{eq:phn}) with additional error terms due to Taylor approximation (the bound holds thanks to the discrete $\ell^2$ stability of the scheme and the consistency of the initial conditions).
Smoothness of the % unique ! unequivocal in its class
reference solution is key to the error estimation.

% Of course, one well-known difficulty with FD methods is that they oscillate % especially in staggered energy-preserving schemes % Gibbs
% and do \emph{not} satisfy a discrete maximum principle. % -- can be consistent yet,if one handles non-positivity in conserv/consistent way

\subsection{Finite-Volume methods} 
\label{sec:FV}

Explicit-in-time Finite-Volume (FV) methods (using time discretization like % RK or
the Euler forward method)
are also standardly used to discretize \eqref{sys:elasto}. % on $\Dcal_T$ with periodic boundary conditions.
% mostly with a view to numerically simulating less regular solutions, % also unique here in linear case
Unlike the FD method above, they provide one with \emph{co-located} cell-averaged approximations $q_{ij}^n$ for $q(t^n,x,y) = (F_1^m, F_2^m, u^m)$ over the control volumes $C_{ij} = (i\Delta\pm\frac{\Delta}2,j\Delta\pm\frac{\Delta}2)$ defined by
$$
q_{ij}^n \approx % \bar q_{ij}(t):=
\frac1{\Delta^2}\int_{C_{ij}} q(t^n,x,y) dxdy.
% \approx q(t^n,i\Delta,j\Delta)
$$
%of the system components $q^l$.
%
The generic FV scheme writes % for \eqref{sys:elasto}
\begin{equation}
\label{FVscheme}
q_{ij}^{n+1} = q_{ij}^n + \frac{k}{\Delta} (\Fcal_{i,j-\frac12} + \Fcal_{i-\frac12,j} - \Fcal_{i,j+\frac12} - \Fcal_{i+\frac12,j})
\end{equation}
which requires formulas for the numerical fluxes $\Fcal_{i\pm\frac12,j}$ and $\Fcal_{i,j\pm\frac12}$  % \approx \int_{t^n}^{t^{n+1}}
approximating the fluxes $F_a(q)$ of the conservation laws \eqref{eq:scl} at cell faces
\begin{align*}
& \Fcal_{i\pm\frac12,j} \simeq \frac{1}{k\Delta}\int_{t^n}^{t^{n+1}}\int_{y_{j-\frac12}}^{y_{j+\frac12}} F_1(q(x_{i\pm\frac12}, y, t)) dydt, \\
& \Fcal_{i,j\pm\frac12} \simeq \frac{1}{k\Delta}\int_{t^n}^{t^{n+1}}\int_{x_{i-\frac12}}^{x_{i+\frac12}} F_2(q(x,y_{j\pm\frac12}, t)) dxdt.
\end{align*}
A standard Two-Points (TP) flux is % so-called Godunov ?
the \emph{upwind} flux
% Godunov/upwind flux, equivalent to
% Roe \cite{(4.61) p.84]{leveque-2002} \cite[4.3.1, Chapter V]{Godlewski-Raviart2021} in linear case
% a kind of Lax-Friedrichs scheme if CFL=1 \cite[Ex.~3.2, Chapter V]{Godlewski-Raviart2021}
% \begin{align}
% \label{LFflux}
% & \Fcal_{n,i+\frac12,j}:=\frac12\left(f_1(q_{n,i+1,j})+f_1(q_{n,ij})\right)-\frac{|A_1|}2\left(q_{n,i+1,j}-q_{n,ij}\right)
% % \approx\frac1{\Delta}\int_{x\in(i\Delta\pm\frac{\Delta}2)} f_1(q(t_n,i\Delta+\frac{\Delta}2,y)) dy
% \\
% & \Fcal_{n,ij+\frac12}:=\frac12\left(f_2(q_{n,ij+1})+f_2(q_{n,ij})\right)-\frac{|A_2|}2\left(q_{n,ij+1}-q_{n,ij}\right)
% \end{align}
% on denoting $|A_a|$ the ``absolute value'' (symmetric !!) matrix
% when the flux vectors are linear functionals $f_a(q)=A_a q$
% !! when nonlinear (Godunov: approximate Riemann solver) 
which yields for each ``acoustic'' sub-system with variable $(F^m_1,F^m_2,u^m)\equiv(v,w,p)$, $m\in\{1,2\}$
%(i.e. $l\in\{1,2,3\}$ or $l\in\{4,5,6\}$)
\begin{align}
\label{eq:vhnFV}
v_{ij}^{n+1} = v_{ij}^n + \frac{kc}{2\Delta} ( & p_{i+1,j}^n-p_{i-1,j}^n 
\\
\nonumber & \quad - 2 v_{ij}^n + v_{i+1,j}^n + v_{i-1,j}^n )
\\
\label{eq:whnFV}
w_{ij}^{n+1} = w_{ij}^n + \frac{kc}{2\Delta} ( & p_{ij+1}^n-p_{ij-1}^n
\\
\nonumber & \quad - 2 w_{ij}^n + w_{ij+1}^n + w_{ij-1}^n )
\\ \nonumber
p_{ij}^{n+1} = p_{ij}^n + \frac{kc}{2\Delta} ( & v_{i+1,j}^n - v_{i-1,j}^n + w_{ij+1}^n - w_{ij-1}^n
\\  & \quad - 4 p_{ij}^n + p_{ij+1}^n + p_{ij-1}^n + p_{i+1,j}^n + p_{i-1,j}^n ) 
\label{eq:phnFV}
\end{align}
where, in comparison with the ``staggered'' scheme (\ref{eq:vhn}--\ref{eq:phn}), 
new terms appear that are not obviously consistent but can be interpreted as artificial diffusion,
small as $O(\Delta^2)$ noticeably.
Now, as is well-known, the latter ``non-consistent'' terms do not prevent convergence: in the case of smooth solutions,
the $O(\Delta^2)$ scaling of those terms even preserves for (\ref{eq:vhnFV}--\ref{eq:phnFV}) the % first-order 
convergence rate observed for (\ref{eq:vhn}--\ref{eq:phn}). A proof is as follows:
% strong
Consistency errors $V_{ij}^n = v_{ij}^n - v(kn,\Delta i,\Delta j)$, $W_{ij}^n = w_{ij}^n - w(kn,\Delta i,\Delta j)$, 
$P_{ij}^n = p_{ij}^n - p(kn,\Delta i,\Delta j)$ % for _strong_ solutions
satisfy
\begin{align}
\label{eq:VhnFV}
V_{ij}^{n+1} & = V_{ij}^n + \frac{kc}{2\Delta} ( P_{i+1,j}^n-P_{i-1,j}^n - 2 V_{ij}^n + V_{i+1,j}^n + V_{i-1,j}^n ) 
+ RV_{ij}^n
\\ 
\label{eq:WhnFV}
W_{ij}^{n+1} & = W_{ij}^n + \frac{kc}{2\Delta} ( P_{ij+1}^n-P_{ij-1}^n - 2 W_{ij}^n + W_{ij+1}^n + W_{ij-1}^n ) 
+ RW_{ij}^n
\\
P_{ij}^{n+1} & = P_{ij}^n + \frac{kc}{2\Delta} ( V_{i+1,j}^n - V_{i-1,j}^n + W_{ij+1}^n - W_{ij-1}^n
\\ & \quad \nonumber - 4 P_{ij}^n + P_{ij+1}^n + P_{ij-1}^n + P_{i+1,j}^n + P_{i-1,j}^n ) 
+ RP_{ij}^n
\label{eq:PhnFV}
\end{align}
using Taylor formulas, with remainder terms $O(k^2)+O(\Delta^2)$ % assuming H^2 solutions
\begin{multline}
RV_{ij}^n = - \int_{t^n}^{t^{n+1}}(t^{n+1}-t)\partial_{tt}^2v(t,i\Delta,j\Delta)dt \\
+ \frac{kc}{2\Delta} \int_{i\Delta}^{(i+1)\Delta}((i+1)\Delta-x)\partial_{xx}^2p(kn,x,j\Delta)dx
- \frac{kc}{2\Delta} \int_{i\Delta}^{(i-1)\Delta}((i-1)\Delta-x)\partial_{xx}^2p(kn,x,j\Delta)dx
\\
+ \frac{kc}{2\Delta} \int_{i\Delta}^{(i+1)\Delta}((i+1)\Delta-x)\partial_{xx}^2v(kn,x,j\Delta)dx
+ \frac{kc}{2\Delta} \int_{i\Delta}^{(i-1)\Delta}((i-1)\Delta-x)\partial_{xx}^2v(kn,x,j\Delta)dx
\end{multline}
\begin{multline}
RW_{ij}^n = - \int_{t^n}^{t^{n+1}}(t^{n+1}-t)\partial_{tt}^2w(t,i\Delta,j\Delta)dt \\
+ \frac{kc}{2\Delta} \int_{j\Delta}^{(j+1)\Delta}((j+1)\Delta-y)\partial_{yy}^2p(kn,i\Delta,y)dy
- \frac{kc}{2\Delta} \int_{j\Delta}^{(j-1)\Delta}((j-1)\Delta-y)\partial_{yy}^2p(kn,i\Delta,y)dy
\\
+ \frac{kc}{2\Delta} \int_{j\Delta}^{(j+1)\Delta}((j+1)\Delta-y)\partial_{yy}^2v(kn,i\Delta,y)dy
+ \frac{kc}{2\Delta} \int_{j\Delta}^{(j-1)\Delta}((j-1)\Delta-y)\partial_{yy}^2v(kn,i\Delta,y)dy
\end{multline}
\begin{multline}
RP_{ij}^n = - \int_{t^n}^{t^{n+1}}(t^{n+1}-t)\partial_{tt}^2p(t,i\Delta,j\Delta)dt
\\
+ \frac{kc}{2\Delta} \int_{(i\Delta}^{(i+1)\Delta}((i+1)\Delta-x)\partial_{xx}^2v(kn,x,j\Delta)dx
- \frac{kc}{2\Delta} \int_{(i\Delta}^{(i-1)\Delta}((i-1)\Delta-x)\partial_{xx}^2v(kn,x,j\Delta)dx
\\
+ \frac{kc}{2\Delta} \int_{(i\Delta}^{(i+1)\Delta}((i+1)\Delta-x)\partial_{xx}^2p(kn,x,j\Delta)dx
+ \frac{kc}{2\Delta} \int_{(i\Delta}^{(i-1)\Delta}((i-1)\Delta-x)\partial_{xx}^2p(kn,x,j\Delta)dx
\\
+ \frac{kc}{2\Delta} \int_{(j\Delta}^{(j+1)\Delta}((j+1)\Delta-y)\partial_{yy}^2w(kn,i\Delta,y)dy
- \frac{kc}{2\Delta} \int_{(j\Delta}^{(j-1)\Delta}((j-1)\Delta-y)\partial_{yy}^2w(kn,i\Delta,y)dy
\\
+ \frac{kc}{2\Delta} \int_{(j\Delta}^{(j+1)\Delta}((j+1)\Delta-y)\partial_{yy}^2p(kn,i\Delta,y)dy
+ \frac{kc}{2\Delta} \int_{(j\Delta}^{(j-1)\Delta}((j-1)\Delta-y)\partial_{yy}^2p(kn,i\Delta,y)dy
\end{multline}
as long as $(v,w,p)$ are $C^2$ smooth solutions.
% Multiplying \eqref{eq:VhnFV}--\eqref{eq:WhnFV}--\eqref{eq:PhnFV} by $V_{ij}^n$,$W_{ij}^n$, $P_{ij}^n$ % respectively
% and summing, the flux terms cancel % on infinite or periodic domain
% and one obtains
% \begin{multline}
% \sum_{ij} \left( |V_{ij}^{n+1}|^2 + |W_{ij}^{n+1}|^2 +  |P_{ij}^{n+1}|^2 \right)
% = 
% \sum_{ij} \left( |V_{ij}^n|^2 + |W_{ij}^n|^2 + |P_{ij}^n|^2 \right)
% \\
% + \sum_{ij} \left( |V_{ij}^{n+1}-V_{ij}^n|^2 + |W_{ij}^{n+1}-V_{ij}^n|^2 + |P_{ij}^{n+1}-V_{ij}^n|^2 \right)
% \\
% + 2\sum_{ij} \left(RV_{ij}^nV_{ij}^n + RW_{ij}^nW_{ij}^n + RP_{ij}^nP_{ij}^n \right)
% \end{multline}
% which allows to conclude to convergence as $k\le \frac1{2c} \Delta \to 0$ % at any fixed time, in l^2 norm
% using \eqref{eq:VhnFV}--\eqref{eq:WhnFV}--\eqref{eq:PhnFV} again to bound
% $$
% \sum_{ij} \left( |V_{ij}^{n+1}-V_{ij}^n|^2 + |W_{ij}^{n+1}-V_{ij}^n|^2 + |P_{ij}^{n+1}-V_{ij}^n|^2 \right)
% $$
% \textbf{ %%%%%%%%%%%%%%%%%%%%%%%%%%%
% by % Young inequality and 
% a ``standard'' weak-BV estimate TBC ?? % this would imply a O(\Delta^\frac12) rate I believed
% Or cancellation finally yield the upwind scheme convergence to smooth solutions at rate $O(\Delta)$ ?
% % better than on unstructured grid see also \cite{Jovanovic2004}
% % similar to implicit \cite{croisille-phd-1990}: RK should not improve anything \cite{S0025-5718-2015-03022-1}
% % optimal rate for upwind in scalar case: \cite{ Richter-OptimalOrderErrorEstimate-1988} 
% {\color{blue} TBC}
% } %%%%%%%%%%%%%%%%%%%%%%%%%%%%%%%%%%
Then, on a square domain $[0,1]^2$ % sufficiently-large to cover compact IC
with periodic boundary conditions say (for the sake of simplicity: see our numerical Section~\ref{sec:numerical} for an illustration when that situation applies), one can % standardly 
study convergence using the Fourier modes $\hat Q^{n}_{gh}$ of the error vector $Q^n_{ij}= \sum_{g,h\in\NN^2} \hat Q^{n}_{gh} e^{2i\pi(ig+hj)}\in\RR^3$
with components $V_{ij}^n = v_{ij}^n - v(kn,\Delta i,\Delta j)$, $W_{ij}^n = w_{ij}^n - w(kn,\Delta i,\Delta j)$, 
$P_{ij}^n = p_{ij}^n - p(kn,\Delta i,\Delta j)$ following Von Neumann. The Fourier modes satisfy % project on Fourier basis 
$$
\hat Q^{n+1}_{gh} = C(g,h) \hat Q^n_{gh} + \hat R^n_{gh}
$$
with $C(g,h)$ the amplification matrix of the linear scheme (\ref{eq:VhnFV}--\ref{eq:WhnFV}--\ref{eq:PhnFV}),
% see e.g. \cite[Chapter 1]{godlewski-raviart-1996}, 
and $C(g,h)$ is power-bounded as long as $k\in (\epsilon_0,\frac2{c})\Delta$ see e.g. \cite{godlewski-raviart-1996,Coulombel2014}: then the upwind scheme is $\ell^2$ % and strongly 
stable and convergence to smooth solutions holds at ``first order'' i.e. % with rate $\Delta$:
$$
\sqrt{\sum_{ij} |Q^N_{ij}|^2 = \sum_{gh} |\hat Q^N_{ij}|^2} \le % sum_n=1^N O(k^2)+O(\Delta^2)
\mathcal{C}_N \Delta
$$
with $\mathcal{C}_N$ independent of $\Delta$, insofar as $R^n_{gh}=O(k^2+\Delta^2)$ and $\sum_{ij} |Q^0_{ij}|^2 = O(\Delta^2)$.
%
% \cite{Vila-Villedieu2003} requires $k\le \frac1{4c} \Delta$ -- but treats non-cartesian grids
% check in Lesaint thesis \cite{Lesaint-1976} ?? (possibly scalar case only)
%
% however it is not ``strongly consistent'' and one cannot conclude to convergence straightforwardly
% despite what \cite{Vila-Villedieu2003} says (before proving convergence O(\Delta^\frac12) for unstructured)

To improve the convergence rate, % on arbitrary meshes ??
one classically considers
\begin{align}
\label{FVschemeHeun}
q_{ij}^{(1)} 
& = q_{ij}^n + \frac{k}\Delta (\Fcal_{i,j-\frac12} + \Fcal_{i-\frac12,j} - \Fcal_{i,j+\frac12} - \Fcal_{i+\frac12,j}),
\\
q_{ij}^{n+1} & = \frac{q_{ij}^n}2 + \frac{q_{ij}^{(1)}}2
 + \frac{k}{2\Delta} (\Fcal_{i,j-\frac12}^{(1)} + \Fcal_{i-\frac12,j}^{(1)} - \Fcal_{i,j+\frac12}^{(1)} - \Fcal_{i+\frac12,j}^{(1)})
\end{align}
based on a second-order time-integrator like RK2 method % SSP
% here applied to the finite volume scheme \eqref{FVscheme}
instead of \eqref{FVscheme}, with numerical fluxes $\Fcal_{i,j\pm\frac12}^{/(1)},\Fcal_{i\pm\frac12,j}^{/(1)}$
consistently improved, i.e. using at cell faces
%
% MUSCL by Van Leer as opposed to ENO  by Harten, Osher, Engquist, Chakravarthy and WENO by Liu, Osher, Chan -- and Shu?
% \begin{equation}
% \label{LFflux2}
% %F_{ij+\frac12}:=
% \frac12\left(A^1q_{ij+{\tfrac12}^+}(t)+A^1q_{ij+{\tfrac12}^-}(t)\right)
% -\frac{|A_1|}2\left(q_{ij+{\tfrac12}^+}(t)-q_{ij+{\tfrac12}^-}(t)\right)
% \end{equation}
% for $\mathcal{F}_{ij+\frac12}$ 
% -- which do not necessarily ensure a higher-order convergence rate for general (smooth) multi-dimensional solutions\footnote{
%     Even without using the -- here useless anyway -- slope-limiters as
%     \begin{multline}
%     q_{ij+{\tfrac12}^+}
%     % = q_{ij+1} - \frac\Delta2\mathop{\rm minmod}\left(\frac{q_{ij+1}-q_{ij}}\Delta,\frac{q_{ij+2}-q_{ij+1}}\Delta\right)  % ,\frac{q_{ij+2}-q_{ij}}{2\Delta}
%     % \\
%     = q_{ij+1} - \mathop{\rm minmod}\left(q_{ij+1}-q_{ij},q_{ij+2}-q_{ij+1}\right)/2  % ,\frac{q_{ij+2}-q_{ij}}2
%     \end{multline}
% } that are non-aligned with the mesh % despite Dumbser JCP ?? -- also: the proof remains a challenge s00574-016-0163-9
%
a piecewise bilinear reconstruction of the fields $q$. Precisely, the numerical fluxes are computed from the value of reconstructed variables $\tilde{q}$ evaluated at cell faces rather than the cell value $q$.
A linear reconstruction $\tilde{q}$ on a cell $C_{ij}$ has the form
\begin{equation} \label{recfct}
\tilde{q}(x,y) = c_0 + c_1\varphi_1(x,y) + c_2\varphi_2(x,y) 
\end{equation}
with $c_0, c_1, c_2$ coefficients to be determinated, using basis functions with zero mean: % on $C_{ij}$
$$
\varphi_1 = \frac{x-x_i}{\Delta}, \quad \varphi_2 = \frac{y-y_j}{\Delta}.
$$
The coefficients $c_k$ are typically defined such that $\tilde{q}$ conserves exactly cell-averaged value $q_0 = q_{ij}^n$ on the central cell $C_0 = C_{ij}$, so $c_0 = q_0$, while it matches some averaged state for a given set $\{C_l\}_l$ of neighbour cells % of $C_0$
in a {\em least-square sense}: %. In other words
$$
\tilde{q} = {\rm arg\,min} \sum_{C_l}\left(q_l - \frac{1}{\Delta^2}\int_{C_l}\tilde{q}dxdy\right)^2.
$$
Consequently, the coefficients $c_1$ and $c_2$ are least-square solution of the (overdeterminated) linear system
 $$
 \sum_{k=1}^2 c_k \bar{\varphi}_{kl} = q_l - q_0 
 $$  
with $\bar{\varphi}_{kl} =  \frac{1}{\Delta^2}\int_{C_l}\varphi_kdxdy$, i.e. the averaged value of the basis function $\varphi_k$ over $C_l$. Introducing the matrix $\mathbf{A}$ whose coefficients are $A_{kl} = \bar{\varphi}_{kl}$ and denoting $\mathbf{B} = (\mathbf{A}^T\mathbf{A})^{-1}\mathbf{A}^T$ (easily pre-computed for the case of Cartesian mesh) one obtains
%the coefficents
\begin{equation}\label{reccoef}
c_k  = \sum_{l}B_{kl} (q_l - q_0).
\end{equation}
In the D2Q5 case with neighbour cells $\{C_l\}_l = \{C_1, C_2, C_3, C_4\}$ we obtain
\begin{gather*}
\mathbf{A} = \begin{pmatrix}
1 & 0 \\
0 & 1 \\
-1 & 0 \\
0 & -1
\end{pmatrix}, \quad 
\mathbf{B} = \begin{pmatrix}
\frac12 & 0 & -\frac12 & 0 \\
0 & \frac12 & 0 & -\frac12 \\
\end{pmatrix} , \\
c_1 = \frac{q_1 - q_3}{2}, \quad c_2 = \frac{q_2 - q_4}{2}.
\end{gather*}
The resulting linear reconstruction $\tilde{q}$ is none other than the classical centered discretization of gradient in each direction.

\end{document}